\documentclass{amsart}[11pt]

\usepackage{amssymb,mathrsfs,amscd,graphicx,color,enumerate, stmaryrd}

\DeclareMathOperator{\Perm}{Perm}
\newcommand{\abs}[1]{\lvert #1 \rvert}
\newcommand{\umod}[1]{(\mathbb{Z}/ #1\, \mathbb{Z})^\times}
\newcommand{\fl}[1]{\left\lfloor #1 \right\rfloor}
\newcommand{\dangle}[1]{\left\langle #1 \right\rangle}
\newcommand{\zmod}[1]{\mathbb{Z}/ #1\, \mathbb{Z}}

\newcommand{\dbracket}[1]{\left\llbracket #1 \right\rrbracket}

\def\odd{\mathrm{odd}}
\def\nw{\mathrm{new}}
\def\old{\mathrm{old}}

\DeclareMathSymbol{\twoheadrightarrow} {\mathrel}{AMSa}{"10}

\DeclareMathOperator{\ord}{ord}

\DeclareMathOperator{\ddiv}{div}

\DeclareMathOperator{\Res}{Res}

\DeclareMathOperator{\Aut}{Aut}
\DeclareMathOperator{\End}{End}

\DeclareMathOperator{\Id}{Id}
\DeclareMathOperator{\Gal}{Gal}
\DeclareMathOperator{\Mat}{Mat}

\DeclareMathOperator{\Tr}{Tr}

\DeclareMathOperator{\Lie}{Lie}
\DeclareMathOperator{\GL}{GL}

\DeclareMathOperator{\CSp}{CSp}

\newcommand{\jj}{\mathfrak{j}}
\newcommand{\ii}{\mathfrak{i}}

\newcommand{\CV}{\mathfrak{C}}

\newcommand{\SV}{\mathfrak{S}}

\newcommand{\SZ}{\mathscr{S}}
\newcommand{\TZ}{\mathscr{T}}

\newcommand{\VZ}{\mathscr{V}}

\newcommand{\Q}{\mathbf{Q}}

\newcommand{\cc}{\mathbb{C}}

\newcommand{\nn}{\mathbb{N}}

\newcommand{\pp}{\mathbb{P}}
\newcommand{\qq}{\mathbb{Q}}
\newcommand{\rr}{\mathbb{R}}
\newcommand{\zz}{\mathbb{Z}}

\newcommand{\EE}{\mathcal{E}}

\newcommand{\OO}{\mathcal{O}}

\def\Sn{{\mathbf S}_n}
\def\An{{\mathbf A}_n}

\newtheorem{thm}{Theorem}[section]
\newtheorem{lem}[thm]{Lemma}
\newtheorem{cor}[thm]{Corollary}
\newtheorem{prop}[thm]{Proposition}
\theoremstyle{definition}
\newtheorem{defn}[thm]{Definition}

\newtheorem{exs}[thm]{Examples}

\newtheorem{rem}[thm]{Remark}

\newtheorem{sect}[thm]{}
\newtheorem{step}{Step}

\numberwithin{equation}{section}

\newtheoremstyle{notitle}  
  {}
  {}
  {\itshape}
  {}
  {}
  {\ }
  {.5em}
  {}
\theoremstyle{notitle}

\title[Endomorphism algebra of Jacobian]{Endomorphism algebras of
  factors of certain hypergeometric Jacobians}

\author{Jiangwei Xue, Chia-Fu Yu}

\address{ Institute of Mathematics, Academia Sinica, 6F,
  Astronomy-Mathematics Building, No. 1, Sec. 4, Roosevelt Road,
  Taipei 10617, TAIWAN, R.O.C. }

\email{xue\_j\char`\@math.sinica.edu.tw, chiafu\char`\@math.sinica.edu.tw}

\begin{document}
\date{\today}

\begin{abstract}
   We classify the endomorphism algebras of factors of the Jacobian of
   certain hypergeometric curves over a field of characteristic
   zero. Other than a few exceptional cases, the endomorphism algebras
   turn out to be either a cyclotomic field $E=\mathbb{Q}(\zeta_q)$, or
   a quadratic extension of $E$, or $E\oplus E$. This result may be
   viewed as a generalization of the well known results of the
   classification of endomorphism algebras of elliptic curves over
   $\mathbb{C}$.
 \end{abstract}
\maketitle

\section{Introduction}
Throughout this paper, the word ``curve'' is reserved for smooth
projective curves, and $N\in \nn$ denotes an integer strictly greater
than 1. If $N$ is a prime power $p^r$, we write $q$ for it instead.
Let $k$ be a field of characteristic zero with algebraic closure
$\bar{k}$, and $A$ be an elliptic curve over $k$.  It is a classical
result that the absolute endomorphism algebra
$\End^0(A):=\End_{\bar{k}}(A)\otimes_\zz \qq$ of $A$ is either $\qq$
or an imaginary quadratic field (cf. \cite[Theorem 5.5]{MR2514094}).
Over characteristic zero, every elliptic curve is defined by a
Weierstrass equation of the form
\begin{equation}
  \label{eq:5}
C_{f,2}: \quad y^2=f(x),
\end{equation}
where $f(x)\in k[x]$ is a polynomial of degree $3$ without multiple
roots.  It is very tempting to replace the exponent of $y$ in
(\ref{eq:5}) by $N$ and study the curve 
\begin{equation}
  \label{eq:8}
  C_{f,N}: \quad  y^N=f(x)
\end{equation}
and its Jacobian variety $J(C_{f,N})$. We are interested
in the endomorphism algebra of
$J(C_{f,N})$.  

There are multiple ways of putting $C_{f,N}$ in a slightly more
general context. In one direction (say $k=\cc$), we may look at the
hypergeometric curves $C_{\lambda, N}$ defined by
\begin{equation}
  \label{eq:hypergeometric_curve}
 y^N=x^A(x-1)^B (x-\lambda)^C,   
\end{equation}
 where $\lambda\in \cc-\{0,
1\}$. These curves are closely related to Gauss's hypergeometric
series $F(a,b,c; z)$ (cf. \cite{MR931211} by J. Wolfart).  Assume that
$A=B=C=1$. One may study the exceptional set
\[\mathscr{E}_N=\{ \lambda \in \cc-\{0,1\}\mid \text{ The Jacobian } J(C_N(\lambda)) \text{ has
  complex multiplication.}\}\] When $N=5$ or $7$, De Jong and Noot
\cite{MR1085259} showed that $\mathscr{E}_N$ is infinite, and thus providing
counter examples (in genus $g=4$ and $6$ respectively) to the
Coleman's conjecture (\cite[Conjecture 6]{MR948246}), which predicted
that for each fixed $g\geq 4$, there are only finitely number of
isomorphic classes of curves of genus $g$ whose Jacobians have complex
multiplication.  Coleman's conjecture remains to be open for $g\geq
8$. See \cite{MoonenOort} for a survey.
 Nowadays, the question of the finiteness of $\mathscr{E}_N$ is
generally seen in the light of Andr\'e--Oort Conjecture (\cite[p.~215,
problem 1]{MR990016},\cite{MR1472499}, \cite[Conjecture
1.5]{MR2166087}), which is a conjecture on the special points of
Shimura varieties.

Another general setting for $C_{f,N}$ is to allow the degree of $f(x)$
in (\ref{eq:8}) to be an arbitrary number $n\geq 3$, while still
requiring that $f(x)$ has no multiple roots.  We call such curves
\textit{superelliptic curves}. In a series of papers \cite{MR2166091},
\cite{MR2471095}, etc., Yu. G. Zarhin determined the endomorphism
algebras $\End^0(J(C_{f,N}))$, assuming that $n\geq 5$, $N=q=p^r$ is a
prime power coprime to $n$, and $f(x)$ is irreducible over $k$ with
Galois group $\Gal(f)$ equal to either the full symmetric group $\Sn$
or the alternating group $\An$ (cf. \cite[Theorem 1.1]{MR2166091},
\cite[Theorem 1.1]{MR2471095}).

To explain Zarhin's results more clearly, and to state our main
theorem, we need to introduce some new concepts.  Clearly, $C_{f,N}$
(with an arbitrary $f(x)$) admits a natural automorphism of order $N$:
\[ \delta_N: C_{f,N}\to C_{f,N}, \qquad (x,y)\mapsto (x,\xi_N y),\]
where $\xi_N\in \bar{k}$ is a primitive $N$-th root of unity. By
Albanese functoriality, $\delta_N$ induces an automorphism of
$J(C_{f,N})$, which will be denoted again by $\delta_N$ by an abuse of
notation.  Thus we obtain an embedding of the cyclic group
$G=\zmod{N}$ into $\Aut_{\bar{k}}(J(C_{f,N}))$, and hence a
homomorphism from the group ring $\qq[G]$ to $\End^0(J(C_{f,N}))$. Let
$\zeta_N:=e^{2\pi i/N}\in \cc$, and $\zeta_D:=\zeta_N^{N/D}$ for each
positive $D\mid N$. The natural isomorphism
\[ \qq[G]\cong \qq[T]/(T^N-1)\cong \prod_{D\mid N} \qq(\zeta_D)\]
gives rise to an isogeny
\begin{equation}
  \label{eq:11}
  J(C_{f,N})\sim \prod_{D\mid N,\, D\neq 1} J_{f, D}^\nw. 
\end{equation}
Each $J_{f,D}^\nw$ is isogenous to an abelian subvariety of
$J(C_{f,D})$, and $J_{f,N}^\nw$ is the abelian subvariety of
$J(C_{f,N})$ that has not appeared in $J(C_{f,D})$ for any proper
divisor $D$ of $N$ before, and thus the name ``new part'' (See
Section~\ref{sec:superelliptic-curves} for more details). We have
naturally an embedding \[\ii:
\zz[\zeta_N]\hookrightarrow \End_{\bar{k}}( J_{f,N}^\nw), \qquad
\zeta_N \mapsto \delta_N\mid_{J_{f,N}^\nw}.\] Zarhin showed that under
afore mentioned assumptions, the embedding $\ii$ is in fact an
isomorphism, and (recall that $q=p^r$)
\begin{equation}
 \label{eq:1}
   \End^0(J(C_{f,q}))\cong \prod_{i=1}^r \qq(\zeta_{p^i}).    
\end{equation}
He also treated the case $n=3, 4$ assuming some further conditions on
the base field and $\Gal(f)$(cf. \cite[Theorem 1.3]{MR2349666}). If
$q=2^r$, and $\deg f(x)=3$, then (\ref{eq:1}) need to be modified
accordingly \cite[Theorem 1.4]{MR2349666}. Many parts of this paper
are based on his results.



We will improve Zarhin's result by removing the extra assumptions and
classify $\End^0(J^\nw_{f,q})$ for all polynomials $f(x)$ of degree
$3$ with nonzero discriminants. Partial results are also obtained for
a general $N$.

The genus formula for $C_{f,N}$ for a general $f(x)$ and $N$ is given
in \cite{MR1107394} and \cite{MR1357406}. 
By (\ref{eq:14}) of Section~\ref{sec:superelliptic-curves}, if $f(x)$
is of degree 3 with nonzero discriminant, and $N>3$, then $\dim J_{f,N}^\nw= \varphi(N)$.
Therefore,
\[\dim_\qq(\qq(\zeta_N))=\varphi(N)=\dim J_{f,N}^\nw.\]
In a way, $J_{f,N}^\nw$ generalizes the elliptic curves in the sense
that they are abelian varieties naturally equipped with multiplication
by cyclotomic fields whose degree coincides with the dimension of the
 variety.

\begin{thm}[Main Theorem]\label{thm:main}
  Let $k$ be an field of characteristic zero, $q=p^r$ be a prime
  power, and $q\geq 9$ if $p=3$, and $q\geq 4$ if $p=2$.  Let $f(x)\in
  k[x]$ be a polynomial of degree $3$ with no multiple roots, and
  $J_{f,q}^\nw$ be defined as in
  Definition~\ref{defn:def-of-new-part}.  Then one of the following
  holds for $J_{f,q}^\nw$:
  \begin{enumerate}
  \item $J_{f,q}^\nw$ is absolutely simple, and $\End^0(J_{f,q}^\nw)$ is
    one of the following,
    \begin{enumerate}
    \item[(1a)] $\End^0(J_{f,q}^\nw)\cong\qq(\zeta_q)$. 
    \item[(1b)] $\End^0(J_{f,q}^\nw)\cong L$, where $L$ is a CM-field containing
      $\qq(\zeta_q)$, and\\ $[L:\qq(\zeta_q)]=2$.
    \end{enumerate}
  \item $J_{f,q}^\nw$ is not absolutely simple, and
    $\End^0(J_{f,q}^\nw)$ is one of the following,
    \begin{enumerate}
     \item[(2a)] $\End^0(J_{f,q}^\nw)\cong \qq(\zeta_q)\oplus
       \qq(\zeta_q)$, if $p\geq 5$, and $q\neq 5, 7$. 
     \item[(2b)] $\End^0(J_{f,q}^\nw)\cong \qq(\zeta_q)\oplus \qq(\zeta_q)$ or
       $\Mat_2(\qq(\zeta_q))$  if $q=3^r\geq 27$. 
    \item[(2c)]  $\End^0(J_{f,q}^\nw)\cong  \Mat_2(\qq(\zeta_q))$ if
      $q=4, 5$ or
      $9$. 
     \item[(2d)] $\End^0(J_{f,q}^\nw)\cong \Mat_3(\qq(\sqrt{-7}))\oplus
       \qq(\zeta_7)$ if $q=7$. 
     \item[(2e)] $\End^0(J_{f,q}^\nw) \cong \Mat_2(\qq(\sqrt{-1}))\oplus
       \Mat_2(\qq(\sqrt{-2}))$ if $q=8$. 
     \item[(2f)] $\End^0(J_{f,q}^\nw) \cong \qq(\zeta_q)\oplus
       \qq(\zeta_q)$ or $\Mat_2(\qq(\alpha))\oplus
       \qq(\zeta_q)$ if $q=2^r\geq 16$, where $\alpha=
       2\sqrt{-1}\sin(2\pi/q))$.
        \end{enumerate}
  \end{enumerate}
In particular,  $\End^0(J(C_{f,q}))=
\prod_{i=1}^r \End^0(J_{f,p^i}^\nw)$ is commutative if $p>7$. 
\end{thm}

\begin{thm}\label{thm:galois-group-is-s3}
  In addition to the assumptions of Theorem~\ref{thm:main}, we assume
  that $k$ contains a primitive $q$-th root of unity $\xi_q$, and
  $f(x)$ is irreducible over $k$ with Galois group $\Gal(f)\cong
  \mathbf{S}_3$, then $J_{f,q}^\nw$ is absolutely simple. In other
  words, either (1a) or (1b) in Theorem~\ref{thm:main} holds, and
  case (2) does not appear.
\end{thm}
The proofs of Theorem~\ref{thm:main} and
Theorem~\ref{thm:galois-group-is-s3} will be given in
Section~\ref{sec:compl-abel-vari}. 
\begin{rem}
  When $q=3$, $C_{f,3}$ has genus 1, and $J_{f,3}^\nw=J(C_{f,3})$ is
  an elliptic curve with $\End_{\bar{k}}(J(C_{f,3}))\supseteq \zz[\zeta_3]$, the
  maximal order in the imaginary quadratic field
  $\qq(\zeta_3)$. Therefore, $\End_{\bar{k}}(J(C_{f,3}))=\zz[\zeta_3]$. Since
  the class number of $\qq(\zeta_3)$ is one, $J(C_{f,3})$ is
  isomorphic over $\bar{k}$ to the elliptic curve $y^2=
  x^3+1$. 
\end{rem}
\begin{rem}
  If $q=4$, then
  $\End^0(J_{f,4}^\nw)=\Mat_2(\qq(\sqrt{-1}))$ for all $f(x)\in k[x]$ of
  degree 3 with no multiple roots. In other words, $J_{f,4}^\nw$ is
  isogenous to the square of elliptic curves $y^2=x^3-x$. This result
  was first proven by J.W.S. Cassels \cite{MR807702}, and an explicit
  construction of the isogeny is given by J. Gu{\`a}rdia in \cite{MR1853452}. 
\end{rem}

\begin{cor}\label{cor:automorphism}
  Let the assumptions be the same as Theorem~\ref{thm:main}. We
  further assume that $p>7$ and $f(x)$ is a monic polynomial. With a
  unique change of variable of the form $x\mapsto x-b$ for a suitable
  $b\in k$, we may assume that $f(x)= x^3+B_0x +C_0$.
  \begin{enumerate}
  \item If $B_0=0$, then $\Aut_{\bar{k}}(C_{f,q})\cong \zmod{3q}$;
  \item If $C_0=0$, then $\Aut_{\bar{k}}(C_{f,q})\cong \zmod{2q}$;
  \item $\Aut_{\bar{k}}(C_{f,q})\cong \zmod{q}$ otherwise.
  \end{enumerate}
In particular, if $f(x)=x(x-1)(x-\lambda)$, then
$\Aut_{\bar{k}}(C_{\lambda,q})\cong\zmod{3q}$ if and only if $\lambda =
(1\pm \sqrt{-3})/2$, and $\Aut_{\bar{k}}(C_{\lambda,q})\cong\zmod{2q}$ if
and only if $\lambda\in \{ -1, 2, 1/2\}$. There are only finitely many
$\lambda\in \cc$ such that the curve $C_{\lambda,q}$ has extra
automorphisms. 
\end{cor}

\begin{rem}
  We will also obtain some results for a general $N$ with $\gcd(N,3)=1$.
  For example, if  there exists a quadratic field extension
  $L/\qq(\zeta_N)$ such that
  \[\End^0(J_{f,N}^\nw)\supseteq L \supseteq \qq(\zeta_N),\]
  and $N\not\in \{4, 10\}$, then it is shown  in
  Corollary~\ref{cor:nw-factor-of-jacobian-quad-ext} that
  $\End^0(J_{f,N}^\nw)$ coincides with $L$, and $J_{f,N}^\nw$ is
  absolutely simple.

\end{rem}

\begin{exs}
  Here are some examples of $f(x)$ which give rise to the endomorphism
  algebras in Theorem~\ref{thm:main}. Most of the proofs will be given
  in Section~\ref{sec:autom-constr-exampl}.
  \begin{enumerate}
  \item[(1a)$\phantom{'}$] Suppose that $k=\bar{\qq}(t)$, the rational
    function field of transcendental degree 1 over $\bar{\qq}$, and
    $f(x)=x^3-x-t\in k[x]$. By Example 2.3 and Theorem 5.18 of
    \cite{MR2349666}, $\End_{\bar{k}}(J_{f,q}^\nw)\cong \zz[\zeta_q]$
    for any prime power $q=p^r\neq 4$.
\item[(1b)$\phantom{'}$] If $f(x)=x^3+1$ and $3\nmid N$, then
  $\End^0(J_{f,N}^\nw)\cong 
  \qq(\zeta_{3N})$. 
\item[(1b')] If $f(x)=x^3-x$ and $N$ is even and coprime to 3, then
  $\End^0(J_{f,N}^\nw)\cong \qq(\zeta_{2N})$.
\item[(2a)$\phantom{'}$] If $f(x)=x^3+x$ and $p\geq 5$, and $q\neq 5, 7$, then
  $\End^0(J_{f,q}^\nw)\cong \qq(\zeta_q)\oplus \qq(\zeta_q)$.
\item[(2b)$\phantom{'}$] If $f(x)=x^3+1$ and $q=3^r\geq 9$, then
  $\End^0(J_{f,q}^\nw)\cong 
  \Mat_2(\qq(\zeta_q))$. 
\item[(2b')] If $f(x)=x^3-x$ and $q=3^r\geq 27$, then
  $\End^0(J_{f,q}^\nw)\cong 
  \qq(\zeta_q)\oplus \qq(\zeta_q)$. 
\item[(2c)$\phantom{'}$] If $f(x)=x^3-x$ and $q=5, 9$, then $\End^0(J_{f,q}^\nw)\cong 
  \Mat_2(\qq(\zeta_q))$. 
\item[(2c')] If $f(x)=x^3+1$ and $q=9$, then $\End^0(J_{f,9}^\nw)\cong 
  \Mat_2(\qq(\zeta_9))$. 
\item[(2d)$\phantom{'}$] If $f(x)=x^3-x$ and $q=7$, then $\End^0(J_{f,7}^\nw)\cong \Mat_3(\qq(\sqrt{-7}))\oplus
       \qq(\zeta_7)$.
  \end{enumerate}
\end{exs}

\begin{rem}
  The examples show that our classification in Theorem~\ref{thm:main}
  is complete for those $q=p^r$ with $p\geq 5$, in the sense that
  there are examples for each case listed in the theorem for those
  $q$. However, if $q=3^r$, we have yet to find examples where
  $\End^0(J_{f,q}^\nw)$ is quadratic extension of $\qq(\zeta_{3^r})$
  (case (1b)); if $q=2^r$, we don't have examples for which
  $J_{f,q}^\nw$ is not simple (case (2e), (2f)).  The remaining cases
  when $q$ is a power of $2$ or $3$ are again supported by examples.
\end{rem}

The paper is organized as follows. In
Section~\ref{sec:superelliptic-curves}, we study superelliptic curves
$C_{f,N}$ and define the subvariety $J_{f,N}^\nw$ of $J(C_{f,N})$. In
Section~\ref{sec:compl-abel-vari}, we show how the information
extracted from the study of $C_{f,N}$ is used in classifying
$\End^0(J_{f,N}^\nw)$. Certain arithmetic results needed there are
postponed to Section~\ref{sec:arithmetic-results}. In
Section~\ref{sec:autom-constr-exampl}, we study the automorphism group
of $C_{f,N}$ and construct examples with the given endomorphism
algebra in our classification.

\textbf{Acknowledgment:} Major parts of the work was completed when
the first named author was a post-doc at National Center for Theoretic
Science (NCTS), Hsinchu, Taiwan. He would like to express his
gratitude to the generous support of NCTS. He was partially supported
by the grant NSC 101-2811-M-001-078.
The second named author was partially supported by the grants 
NSC 100-2628-M-001-006-MY4 and AS-98-CDA-M01.

\section{superelliptic curve and its
  Jacobian variety}\label{sec:superelliptic-curves}

The goal of this section is to define the abelian subvariety
$J_{f,N}^\nw$ and study its basic properties. Let $k$ be a field with
characteristic coprime to $N$.  We assume that $k$ contains a
primitive $N$-th root of unity $\xi_N$. Then $\xi_D:=(\xi_N)^{N/D}$ is
a primitive $D$-th root of unity for each $D\mid N$. In the latter
half of the section, we will restrict to the case that $k$ has
characteristic zero.


\begin{sect} \label{subsec:general-results-about-curves} We first
  recall some basic fact about curves and their Jacobian
  varieties. Let $X$ be a curve over $k$ of genus $g\geq 2$, and
  $\Aut_k(X)$ its automorphism group over $k$. It is well known that
  $\Aut_{\bar{k}}(X)$ (and hence $\Aut_k(X)$) is finite if $g\geq 2$
  (\cite[Exercise IV.5.2]{MR0463157},\cite[Chapter 11]{MR2386879}). By
  Albanese functoriality, each $\delta\in \Aut_k(X)$ induces an
  automorphism of the Jacobian variety $J(X)$, which is still denoted
  by $\delta$ by an abuse of notation. Torelli's theorem
  (\cite[Section 12]{MR861976}) implies that the homomorphism
  $\Aut_k(X)\to \Aut_k(J(X))$ thus obtained is an embedding. This
  gives rise to a homomorphism from the group ring $\zz[\Aut_k(X)]$ to
  the endomorphism ring of $J(X)$:
  \begin{equation}
    \label{eq:16}
    \zz[\Aut_k(X)]\to \End_k(J(X))\subseteq  \End_{\bar{k}}(J(X)). 
  \end{equation}
 
  Let $\pi: X\to Y$ be a separable map of curves of degree $m$.  It
  induces two morphisms of the Jacobians:
  \begin{align*}
    \pi^*: J(Y)&\to J(X) \qquad \text{ by Picard functoriality};\\
    \pi: J(X)&\to J(Y) \qquad \text{ by  Albanese functoriality}.
  \end{align*}
  Moreover, $\pi\circ \pi^*= m_{J(Y)}$. So $\ker \pi^*\leq J(Y)[m]$,
  the $m$-torsions of $J(Y)$. In particular, $J(Y)$ is isogenous to
  its image $\pi^*J(Y)$:
  \begin{equation}
    \label{eq:15}
    J(Y)\sim \pi^*J(Y)\subseteq J(X). 
  \end{equation}

  Let $G\leq \Aut_k(X)$ be a subgroup of order $m$, and $Y:=X/G$ be
  the quotient curve. The quotient map $\pi: X\to Y$ is separable and
  finite of degree $m$.  We have 
\begin{equation}
  \label{eq:17}
  \big(\sum_{g\in G} g\big) J(X)= \pi^*J(Y).
\end{equation}
In particular, if $Y\cong \pp^1$, then $\sum_{g\in G}
g=0\in \End(J(X))$. 

\end{sect}

\begin{sect}

  Suppose that $f(x)\in k[x]$ is a polynomial of degree $n\geq 3$ with
  factorization $a_0\prod_{i=1}^s(x-\alpha_i)^{m_i}$ in $\bar{k}[x]$,
  and $\gcd(N, m_1, \cdots, m_s)=1$. Let $C_{f,N}$ be the curve
  defined by $y^N=f(x)$ and $J(C_{f,N})$ the Jacobian variety of
  $C_{f,N}$. The map
\begin{equation}
  \label{eq:cyclic_cover}
  \pi : C_{f,N}\to \pp^1,\qquad (x,y)\to x
\end{equation}
realizes $C_{f,N}$ as a (ramified) cyclic cover of $\pp^1$ with
covering group $G:=\Aut(\pi)\cong \zmod{N}$. A generator of
$G$ is given by
\[ \delta_N: C_{f,N}\to C_{f, N}, \qquad (x,y)\mapsto (x, \xi_N
y).\] 

Let $H_D$ be the subgroup of $G$ of index $D$. Then the quotient curve
$C_{f,N}/H_D$ is isomorphic to $C_{f, D}$ with quotient map
\begin{equation}
  \label{eq:quotient}
 \pi_D: C_{f,N}\to C_{f, D}, \qquad (x,y)\mapsto (x, y^{N/D}).  
\end{equation}
\end{sect}

\begin{defn}\label{defn:def-of-new-part}
Following \cite[Definition 5.1]{MR1708603}, we call the abelian
subvariety 
\[ J_{f,N}^\old: = \sum_{D\mid N, D\neq N} \pi_D^* J(C_{f,D}) \] the
\textit{old part} of $J(C_{f,N})$, and its orthogonal complement (with
respect to the canonical polarization) the \textit{new part}
$J_{f,N}^\nw$.  If $N=p$ is a prime, then $J_{f,p}^\nw:=J(C_{f,p})$. 
\end{defn}
\begin{rem}
  If $k=\cc$, $J_{f,N}^\nw$ can also be defined as the complex torus
  given by a period lattice which is obtained by integration on
  $C_{f,N}$. For example, in the case of $C_{f,N}$ is hypergeometric,
  this construction is carried out in \cite[Section 3]{MR1075639}. 
\end{rem}

\begin{sect}
  By \cite[Corollary 5.4]{MR1708603}, 
  \begin{equation}
    \label{eq:14}
\dim J_{f,N}^\nw= \varphi(N)(\abs{R}-2)/2,    
  \end{equation}
  where \[R=\{ P \in \pp^1(\bar{k})\mid \pi: C_{f,N}\to \pp^1 \text{
    is ramified over } P\}.\] Assume that
  $f(x)=a_0\prod_{i=1}^n(x-\alpha_i)$ has no multiple roots. We have
  two cases:
  \begin{itemize}
  \item if $N\nmid n$, then $R=\{\alpha_i\}_{i=1}^n\cup \{ \infty\}$,
    so $\dim J_{f,N}^\nw= \varphi(N)(n-1)/2$;
\item otherwise, $N\mid n$, then $R=\{\alpha_i\}_{i=1}^n$, so $\dim
  J_{f,N}^\nw= \varphi(N)(n-2)/2$.
  \end{itemize}
  Indeed, in the case $n=Nb$, a simple change of variable of the form
  $u=1/(x-\alpha_1)$, $v=y/(x-\alpha_1)^b$ establishes a birational
  isomorphism between $C_{f,N}$ and $C_{g,N}$ over $k':=k(\alpha_1)$,
  where $g(x)\in k'[x]$ is of degree $n-1$ without multiple roots (\cite[Remark 4.3]{MR2166091}).
\end{sect}

\begin{sect}\label{subsec:endormorphism-algebra-contain-cyclotoic-field}
  We show that there exists a natural embedding $\ii:
  \zz[\zeta_N]\hookrightarrow \End_k(J_{f,N}^\nw)$. The main idea of
  the proof is already contained in \cite[Lemma 5.2]{MR1708603}. It is
  included here since the construction is needed for
  Subsection~\ref{subsec:complete-decomposition-isogeny}.  Since $G=
  \zmod{N}$ is commutative,
  \begin{equation}
    \label{eq:group_ring}
 \cc[G]\cong \bigoplus_{\chi\in \widehat{G}(\cc)} \cc_\chi,     
  \end{equation}
  where $\widehat{G}(\cc):=\{\chi: G \to \cc^\times\}$ is the
  ($\cc$-valued) character group of $G$, and $\cc_\chi:=\cc$ with the
  projection map $\cc[G]\to \cc_\chi$ given by $g\mapsto \chi(g)$ for
  all $g\in G$. Let $\chi_N$ be the generator of $\widehat{G}(\cc)$
  with $\chi_N(\delta_N)=\zeta_N:=e^{2\pi i/N}$.  For simplicity, we
  write $\cc_a$ for $\cc_{\chi_N^a}$ for each $a\in \zmod{N}$. Let
  $\epsilon_a\in \cc[G]$ be the element associated with the character
  $(\chi_N)^a$:
\begin{equation}
  \label{eq:idempotent}
  \epsilon_a=\epsilon_{\chi_N^a}:=\frac{1}{N}\sum_{i=0}^{N-1}
   \chi_N(\delta_N)^{ai}\delta_N^{N-i}=\frac{1}{N}\sum_{i=0}^{N-1}
   \zeta_N^{ai}\delta_N^{N-i}\in \cc[G].
\end{equation}
The orthogonality of characters implies that $\epsilon_a$ is mapped to
the primitive idempotent $(0,\cdots, 0, 1,0\cdots, 0)\in
\cc_a$ on the right hand side of
(\ref{eq:group_ring}). Therefore, $\{\epsilon_a\}_{a\in \zmod{N}}$ is
a complete set of primitive pairwise orthogonal idempotents of
$\cc[G]$. For each $D\in \nn$, let $\Phi_D(T)\in \zz[T]$ be the $D$-th
cyclotomic polynomial. We have
\begin{equation}
  \label{eq:cyclotomic_factorization}
  \qq[G]\cong \qq[T]/(T^N-1)\cong \bigoplus_{D\mid N}
  \qq[T]/(\Phi_D(T)),
\end{equation}
and
\begin{equation}
  \label{eq:D-cyclotomic-factor-over-cc}
  \left(\qq[T]/(\Phi_D(T))\right)\otimes_\qq \cc \cong \bigoplus_{a\in
    \umod{D}} \cc_{\chi_D^a}, \qquad \text{where } \chi_D=\chi_N^{N/D}.
\end{equation}
Let  $\zeta_D:=\zeta_N^{N/D}=e^{2\pi i/D}$, and 
\begin{equation}
  \label{eq:21}
 \eta_D:=\sum_{a\in \umod{D}}
\epsilon_{\chi_D^a}=\frac{1}{N}\sum_{i=0}^{N-1}
(\Tr_{\qq(\zeta_D)/\qq}\zeta_D^i) \delta_N^{N-i}\in \qq[G].
\end{equation}
Then $\eta_D$ is the primitive idempotent in $\qq[G]$ corresponding to
the factor $\qq[T]/(\Phi_D(T))$.  In particular,
$\Phi_D(\delta_N)\eta_D=0$.  Clearly, $N\eta_N\in \zz[G]$ and
\cite[Corollary 5.3]{MR1708603} showed that
\begin{equation}
  \label{eq:12}
  J_{f,N}^\nw= (N\eta_N) J(C_{f,N}),\qquad J_{f,N}^\old= N(1-\eta_N)J(C_{f,N}),
\end{equation}
and
\begin{equation}
  \label{eq:20}
  J(C_{f, N})\sim J_{f,N}^\old\times J_{f,N}^\nw. 
\end{equation}
Since $\delta_N$ commutes with $\eta_N$, $G$ acts on $J_{f,N}^\nw$ as
well . So \[\Phi_N(\delta_N) J_{f,N}^\nw = N\Phi_N(\delta_N)\eta_N
J(C_{f,N})=0.\] Therefore, we have a natural embedding
\begin{equation}
  \label{eq:13}
\ii:   \zz[\zeta_N]\hookrightarrow \End_k(J_{f,N}^\nw),\qquad \zeta_N\mapsto \delta_N\mid_{J_{f,N}^\nw}.
 \end{equation}
Similarly, one sees that $(N\eta_D)J(C_{f,N})$ is $G$-invariant
for all $D\mid N$. 
\end{sect}




\begin{sect}\label{subsec:complete-decomposition-isogeny}
  The isogeny in (\ref{eq:20}) can be refined further.  Recall that
  $\{\eta_D\}_{D\mid N}$ form a complete set of primitive pairwise
  orthogonal basis for $\qq[G]$. By the remark below (\ref{eq:17}),
  $(N\eta_1)J(C_{f,N})=0$. So
\begin{equation}
  J(C_{f,N}) \sim \prod_{D\mid N, D\neq 1} (N^2\eta_D)J(C_{f,N}). 
\end{equation}
For each $D\mid
  N$, let \[\tilde{\epsilon}_D:=\frac{1}{\abs{H_D}}\sum_{h\in H_D}
  h=\frac{D}{N}\sum_{h\in H_D} h\in \qq[G].\] By (\ref{eq:17}),
  $(N/D)\tilde{\epsilon}_DJ(C_{f,N})= \pi_D^*J(C_{f,D})$. It was shown
  in \cite[Lemma~5.2]{MR1708603} that $\tilde{\epsilon}_D=\sum_{a\in
    \zmod{D}} \epsilon_{\chi_D^a}$. In particular, $\eta_D
  \tilde{\epsilon}_D =\eta_D$.

Clearly, we have the following commutative diagram of morphism of
curves: 
\begin{equation}
  \label{eq:18}
  \begin{CD}
    C_{f,N}@>{\pi_D}>> C_{f,D}\\
    @V{\delta_N}VV  @VV{\delta_D}V\\
    C_{f,N}@>{\pi_D}>> C_{f,D}\\
  \end{CD}
\end{equation}
So $\pi_D^* \delta_D^*=\delta_N^* \pi_D^*$. Note that
$\delta_N^*=\delta_N^{-1}$, and similarly for $\delta_D$. Hence
\begin{equation}
  \label{eq:19}
  \delta_N \pi_D^*= \pi_D^* \delta_D. 
\end{equation}
It follows that 
\[
\begin{split}
(N^2\eta_D)J(C_{f,N})&= D(N\eta_D)\cdot (N/D)\tilde{\epsilon}_DJ(C_{f,N})=
D(N\eta_D)\pi_D^*J(C_{f,D})\\ 
&=D \left(\sum_{i=0}^{N-1}
(\Tr_{\qq(\zeta_D)/\qq}\zeta_D^i) \delta_N^{N-i}
\right)\pi_D^*J(C_{f,D})\qquad \text{ by (\ref{eq:21}),}\\
 &=D (N/D)\pi_D^*\left(\sum_{i=0}^{D-1}\Tr_{\qq(\zeta_D)/\qq}\zeta_D^i)
   \delta_D^{D-i}\right)J(C_{f,D})   \quad \text{ by (\ref{eq:19}),}\\
&= \pi_D^*J_{f,D}^\nw\sim J_{f,D}^\nw  \quad \text{ by (\ref{eq:15})}. 
\end{split}
 \]
Therefore, 
\begin{equation}
  \label{eq:22}
  J(C_{f,N}) \sim
  \prod_{D\mid N, D\neq 1} J_{f,D}^\nw. 
\end{equation}
This generalizes \cite[Corollary 4.12]{MR2166091}. 
\end{sect}
\begin{rem}
  More generally, let $\pi: C \to C'$ be a cyclic cover of curves with
  covering group $G=\zmod{N}$ (cf. \cite[Definition
  5.1]{MR1708603}). The constructions in
  Subsections~\ref{subsec:endormorphism-algebra-contain-cyclotoic-field}
  and \ref{subsec:complete-decomposition-isogeny} apply without much
  changes.  We see that $\qq(\zeta_N)\hookrightarrow\End(J_C^\nw) $,
  and
\[ J(C)\sim \prod_{D\mid N, D\neq 1}J_{C/H_D}^\nw \times J(C'), \]
since $(N\eta_1)J(C)=\pi^*J(C')$. 
\end{rem}
\begin{sect}
  Let $X$ be a smooth projective curves over $k$. We write $\lambda_X:
  J(X)\to J(X)^\vee$ for the canonical polarization of $J(X)$. It is
  well known that $\lambda_X$ is an isomorphism. Let $\pi: X\to Y$ be
  a morphism of curves, and $\pi: J(X)\to J(Y)$ and $\pi^*: J(Y)\to
  J(X)$ be the induced morphisms of the Jacobians as in
  Subsection~\ref{subsec:general-results-about-curves}. We write
  $\pi^\vee: J(Y)^\vee \to J(X)^\vee$ for the dual homomorphism of
  $\pi$. Then there is a commutative diagram:
  \begin{equation}
    \label{eq:CD-polarization}
    \begin{CD}
    J(Y)@>{\pi^*}>> J(X)\\
    @V{\lambda_Y}V{\cong}V @V{\cong}V{\lambda_X}V\\
    J(Y)^\vee @>{\pi^\vee}>> J(X)^\vee.
    \end{CD}
  \end{equation}
In other words, if we identify each Jacobian with its dual via the
canonical polarization, then $\pi$ and $\pi^*$ are dual to each
other(cf. \cite[Prop.11.11.6]{MR2062673} in the case $k=\cc$, and
\cite[Prop A.6]{MR2987306} in much more generality).  
\end{sect}

\begin{sect}
  Let $i_N : J_{f,N}^\old\to J(C_{f,N})$ be the inclusion, and
  $\lambda_N:=\lambda_{C_{f,N}}:J(C_{f,N})\to J(C_{f,N})^\vee$ be the
  canonical principal polarization.  By
  the proof of \cite[Theorem 19.1]{Mumford_AV}, $J_{f,N}^\nw$ is the
  identity component of $\ker(i_N^\vee\circ \lambda_N)$.  Similar to
  Definition~\ref{defn:def-of-new-part}, we define $J_{f,N}^{D-\nw}$
  to be the orthogonal complement of $\pi_D^*J(C_{f,D})$. Then
  $J_{f,N}^{D-\nw}$ coincides with the identity component of 
\[\ker((\pi_D^*)^\vee\circ \lambda_N)=\ker(\lambda_D^{-1}\circ
(\pi_D^*)^\vee\circ \lambda_N )= \ker \pi_D,  \qquad \text{ by }  (\ref{eq:CD-polarization}).\]

Suppose $D_1\mid D_2$ and $D_2\mid N$, then the quotient map
$\pi_{D_1}: C_{f, N}\to C_{f,D_1}$ factors as a composition of
successive quotient maps $ C_{f,N}\to C_{f, D_2}\to C_{f,D_1}.$
Therefore,
\[ \pi_{D_1}^*J(C_{f, D_1})\subseteq \pi_{D_2}^* J(C_{f, D_2})
\subseteq J(C_{f, N}).\] 
In particular, if $N=q=p^r$ is a prime power, then 
$J_{f,q}^\old= \pi_p^* J(C_{f, q/p})$, and 
\[ J_{f,q}^\nw= J_{f,q}^{(q/p)-\nw}= \text{ identity component of }
\ker \pi_{q/p} .\]
\end{sect}

\begin{sect}\label{sec:connectedness-of-the-kernel}
  We assume that $f(x)\in k[x]$ satisfies one of the following
  conditions:
\begin{itemize}
\item there exists a root $\alpha\in \bar{k}$ of $f(x)$ such that its
  multiplicity $m_\alpha$ is coprime to $N$;
\item $\deg f(x)$ is coprime to $N$. 
\end{itemize}
In the first case, there is exactly one point $P\in C_{f,N}(\bar{k})$
corresponding to $(\alpha, 0)\in \mathbb{A}^2({\bar{k}})$. (Generally
one needs to perform some desingularization to obtain $C_{f,N}$.)
Moreover, the covering map $\pi: C_{f,N}\to \pp^1$ in
(\ref{eq:cyclic_cover}) is totally ramified at $P$. In the second
case, there is exactly one point $P:=\infty$ at infinity for $C_{f,N}$, and
$\pi$ is totally ramified at $P$ again. Either way, it follows that
$\pi_D: C_{f,N}\to C_{f,D}$ is totally ramified at $P$ for each $D\mid
N$.  

Let $K_{N,D}$ be the kernel of $\pi_D^*: J(C_{f,D})\to J(C_{f,N})$. We
have seen in Subsection~\ref{subsec:general-results-about-curves} that
$K_{N,D}\leq J(C_{f,D})[N/D]$. By \cite[Section 9]{MR861976},
$K_{N,D}^\vee(\bar{k})$ is isomorphic to the covering group of the
maximal abelian unramified covering (over $\bar{k}$) of $C_{f,D}$
which is intermediate to $\pi_D: C_{f,N}\to C_{f,D}$.  Such a
covering must be trivial under our assumption on $f(x)$. Therefore,
$K_{N,D}$ is trivial, and $\pi_D^*$ is an embedding for all $D\mid
N$. 

Let $A$ be the quotient abelian variety of $J(C_{f,N})$ by
$\pi_D^*J(C_{f,D})$. We have an exact sequence 
\[  0\to J(C_{f,D})\xrightarrow{\pi_D^*}  J(C_{f,N})\to A\to 0. \]
Taking the dual exact sequence, we get
\[ 0 \to A^\vee \to J(C_{f,N})^\vee \xrightarrow{(\pi_D^*)^\vee}
J(C_{f,D})^\vee \to 0.\] By (\ref{eq:CD-polarization}), we may rewrite
the exact sequence as
\[ 0 \to A^\vee \to J(C_{f,N})\xrightarrow{\pi_D} J(C_{f,D}) \to 0.\]
Therefore, $\ker\pi_D=A^\vee$ is connected.  On the other hand, recall
that $J_{f,N}^{D-\nw}$ is equal to the identity component of $\ker
\pi_D$. It follows that
\begin{equation}
  \label{eq:kernel-connected}
  J_{f,N}^{D-\nw}=\ker \pi_D=A^\vee.
\end{equation}
Since $\pi_D^*$ is an embedding,
\begin{equation}
  \label{eq:torsion-inherence}
 J(C_{f,D})[N/D]= \ker (\pi_D\circ \pi_D^*)\subseteq \ker
\pi_D=J_{f,N}^{D-\nw}. 
\end{equation}
In other words, $\pi_D^* J(C_{f,D})\cap J_{f,N}^{D-\nw}= J(C_{f,D})[N/D]$.
Note that both $\pi_D$ and $\pi_D^*$ are defined over $k$, so
$J_{f,N}^{D-\nw}$ ``inherits'' from $J(C_{f,D})$ a
$\Gal(\bar{k}/k)$-module structure that's isomorphic to
$J(C_{f,D})[N/D]$.

In particular, if $N=q=p^r$ is a prime power, and $f(x)$ has no
multiple roots, then $J_{f,q}^\nw=\ker \pi_{q/p}$, and
\begin{equation}
  \label{eq:prime-power-inherence-torsion}
J_{f,q}^\nw[p]\supseteq \pi_{q/p}^* J(C_{f,q/p})[p] \cong
J(C_{f,q/p})[p]. 
\end{equation}
We have an exact sequence 
\[ 0 \to J(C_{f,q/p})[p]\to J(C_{f, q/p})\times J_{f,q}^\nw\to
J(C_{f,q})\to 0, \]
which makes (\ref{eq:20}) more explicit. 
\end{sect}
\begin{sect}\label{subsec:Serre-duality-differential}
  Let $X$ be a curve over $k$, and $\Lie_k(J(X))$ be
  the Lie algebra of $J(X)$, which is canonically isomorphic to the
  tangent space to $J(X)$ at $0$. The Picard functoriality induces a
  right action of $\Aut_k(X)$ on $\Lie_k(J(X))$. The
  isomorphism $\Lie_k(J(X))\cong H^1(X, \OO_X)$ given in
  \cite[Proposition 2.1]{MR861976} is
  $\Aut_k(X)$-equivariant. Combining with the Serre duality
  \cite[Corollary 7.13]{MR0463157}), we obtain a perfect and
  $\Aut_k(X)$-equivariant pairing
\begin{equation}
  \label{eq:pairing-Lie-alg-differential}
  \Gamma(X, \Omega_X^1)\times \Lie_k(J(X))\to k,
\end{equation}
where $\Aut_k(X)$ acts on $\Gamma(X, \Omega_X^1)$ from the right via
pull-backs. (Over $\cc$, this follows directly from the classical
definition of the Jacobian \cite[Section 2]{MR861976}.) Note that
$\Aut_k(X)$ also acts on $\Lie_k(J(X))$ from the left via Albanese
functoriality, which is just the inverse of the Picard
action. Therefore, we will also let $\Aut_k(X)$ act on $\Gamma(X,
\Omega_X^1)$ from the left by taking the inverse of the pullback so
that (\ref{eq:pairing-Lie-alg-differential}) is again
$\Aut_k(X)$-equivariant.
\end{sect}

\begin{sect}
  Since $k$ contains a primitive $N$-th root of unity $\xi_N$, any
  left representation $V$ of $G=\zmod{N}$ over $k$ splits into a
  direct sum of subrepresentations, indexed by the $k$-valued
  character group $\widehat{G}(k)$ of $G$.
  \begin{equation}
    \label{eq:23}
    V= \bigoplus_{\chi\in \widehat{G}(k)} V_\chi, 
  \end{equation}
  where $V_\chi:=\{ v\in V\mid gv = \chi(g) v, \forall g\in G\}$.

  Recall that $J_{f,N}^\nw = N\eta_N J(C_{f,N})$. Let $d(N\eta_N):
  \Lie_k(J(C_{f,N}))\to \Lie_k(J(C_{f,N}))$ be the induced morphism of
  Lie algebras of $N\eta_N\in \End(J(C_{f,N}))$.  Since the isogeny
  $J_{f,N}^\old \times J_{f,N}^\nw \to J(C_{f,N})$ in (\ref{eq:20}) is
  separable with kernel isomorphic to a subgroup of $J(C_{f,N})[N]$,
  $\Lie_k(J_{f,N}^\nw)$ coincides with the image of $d(N\eta_N)$.
  Clearly, each $\Lie(J(C_{f,N}))_\chi$ is $d(N\eta_N)$ invariant. It
  follows from (\ref{eq:21}) that $d(N\eta_N)$ acts on
  $\Lie(J(C_{f,N}))_\chi$ as multiplication by $N$ if $\chi(\delta_N)$
  is a primitive $N$-th root of unity in $k$, and $0$
  otherwise. Therefore, 
  \begin{equation}
    \label{eq:24}
    \Lie_k(J_{f,N}^\nw)= \bigoplus_{a\in \umod{N}} \Lie_k(J(C_{f,N}))_{\chi_N^a},
  \end{equation}
  where $\chi_N$ is the unique character in $\widehat{G}(k)$ such that
  $\chi_N(\delta_N)=\xi_N$.  We write $h: \umod{N}\to \nn$ for the
  dimension function define by 
\begin{equation}
  \label{eq:25}
h(a)= \dim_k \Lie_k(J(C_{f,N}))_{\chi_N^a}.
\end{equation}

\begin{sect}\label{subsec:explicit-multiplicity-differential} Assume that $k$ has characteristic zero, and
  $\xi_N=\zeta_N=e^{2\pi i/N}$.  We force $G\cong \zmod{N}$ to act on
  $\Gamma(C_{f,N}, \Omega_{C_{f,N}}^1)$ from the left by taking the
  inverse of the pullback.  It follows from
  Subsection~\ref{subsec:Serre-duality-differential} that
  \[h(a)= \dim_k \Gamma(C_{f,N}, \Omega_{C_{f,N}}^1)_{\chi_N^{-a}}.\]
  If $f(x)$ has no multiple roots and $N\nmid n$, then 
  \begin{equation}
    \label{eq:basis-of-diff-first-kind}
 \left\{\frac{x^{b-1}dx}{y^a}\mid 1\leq a<N, 1\leq b\leq
    \fl{\frac{na}{N}} \right\}
  \end{equation}
  is a basis for $\Gamma(C_{f,N}, \Omega_{C_{f,N}}^1)$ by
  \cite[Proposition 2]{MR1357406}. Clearly, each $x^{b-1}dx/y^a$ is an
  eigenvector for $(\delta_N^{-1})^*$ corresponding to eigenvalue
  $\zeta_N^a$.
  In particular, if $f(x)$ has no multiple roots, and $N\nmid n$, then
     \begin{equation}
       \label{eq:26}
       h(a)=
       n-1-\fl{\frac{na}{N}}.
     \end{equation}
  Here $\fl{t}$ is the smallest integer less or equal to $t$ (i.e.,
  the floor function). In the floor function of (\ref{eq:26}), we
  take $a$ to be the unique integer between 0 and $N-1$ for the
  corresponding residue class. One easily checks that $h(a)+h(-a)=
  n-1$ for the function $h$ in (\ref{eq:26}).
  \end{sect}


  Let $E$ be a number field, and $k$ be a field of characteristic zero
  that contains all conjugates of $E$. Let $\Sigma_E^k=\{ \sigma \mid
  \sigma: E\hookrightarrow k\}$ be the set of all embeddings of $E$
  into $k$. (We'll drop the superscript $k$ if $k=\cc$).  Any $(E,
  k)$-bimodule $V$ splits into a direct sum of $k$-vector spaces
  $V=\oplus_{\sigma \in \Sigma_E^k} V_\sigma$, where $V_\sigma:=\{
  v\in V \mid e\cdot v = \sigma(e)v, \forall e\in E\}$. Mimicking the
  definition of CM-types, we make the following definition.

  \begin{defn} 
    Let $E$ and $k$ be as above. Suppose that $(X, \ii)$ is a pair
    consisting an abelian variety $X/k$ together with an embedding $\ii:
    E\hookrightarrow \End_k^0(X)$. Then $\Lie_k(X)$ is naturally an $(E,
    k)$-bimodule. The function $h: \Sigma_E^k\to \nn$ defined by 
\[ h(\sigma)= \dim_k \Lie_k(X)_\sigma\]
is called the \textit{generalized multiplication type} of $(X, \ii)$. 
  \end{defn}

  \begin{sect}\label{subsec:def-generalized-mul-type}
    Let the assumptions be the same as
    Subsection~\ref{subsec:explicit-multiplicity-differential}. Consider
    the pair $(J_{f,N}^\nw, \ii)$ with $\ii:
    \qq(\zeta_N)\hookrightarrow \End^0( J_{f,N}^\nw)$ given in
    (\ref{eq:13}). Then
    \[ \Sigma_{\qq(\zeta_N)}^k = \{ \sigma_a\mid a\in \umod{N}, \text{
      and } \sigma_a : E\hookrightarrow k, \quad \zeta_N\mapsto
    \xi_N^a\}, \] which is naturally identified with the set
    $\umod{N}$. One easily sees
    that \[\Lie_k(J_{f,N}^\nw)_{\sigma_a}=\Lie_k(J_{f,N}^\nw)_{\chi_N^a}.\]
    Therefore, the generalized multiplication type of $(J_{f,N}^\nw,
    \ii)$ is given by (\ref{eq:26}) under the aforementioned
    assumptions on $f(x)$ and $N$.
  \end{sect}

  \begin{sect}
    Let $\iota: \qq(\zeta_N)\to \qq(\zeta_N)$ be the complex
    conjugation, $\bar{\ii}:=\ii\circ \iota$, and $\bar{h}$ be the
    multiplication type of $(J_{f,N}^\nw, \bar{\ii})$. Then 
    \begin{equation}
      \label{eq:33}
      \bar{h}(a)=h(-a)= \dim_k \Gamma(C_{f,N},
      \Omega_{C_{f,N}}^1)_{\chi_N^a}. 
    \end{equation}
    This saves us the trouble to go from the left representation
    $\Gamma(C_{f,N}, \Omega_{C_{f,N}}^1)$ of $G$ to its dual
    representation $\Lie_k(J(C_{f,N}))$ in some
    calculations. Moreover, (\ref{eq:26}) takes a simpler form
    \begin{equation}
      \label{eq:34}
      \bar{h}(a)= \fl{\frac{na}{N}}. 
    \end{equation}
Therefore, it is more convenient to replace $(\ii, h)$ with
$(\bar{\ii}, \bar{h})$, which we will do in the next section. 
  \end{sect}


\end{sect}
\section{complex abelian varieties with given multiplication type}
\label{sec:compl-abel-vari}
Throughout this section, $E:=\qq(\zeta_N)$ is the $N$-th cyclotomic
field, and $(X, \ii)$ will denote a pair consisting a complex abelian
variety $X$ together with an embedding $\ii:
E\hookrightarrow \End^0(X)$. We will identify $E$ with its image in
$\End^0(X)$ via $\ii$ and write $E\subseteq \End^0(X)$.  Let $h:
\umod{N}\to \nn$ be the generalized multiplication type of $(X, \ii)$.
We will classify $\End^0(X)$, using arithmetic properties of $h$. In
the case $X=J_{f,N}^\nw$, we assume that $h$ is given by
(\ref{eq:34}).

\begin{sect}\label{subsec:Hodge-decomposition}
  In general, let $\EE$ be a number field, and $g: \Sigma_\EE\to \nn$
  be the generalized multiplication type of a pair $(Z, \jj)$ of a
  complex abelian variety $Z$ together with an embedding $\jj:
  \EE\to \End^0(Z)$.  The first rational homology group
  $V_\qq:=H_1(Z,\qq)$ carries naturally a structure of faithful
  $\End^0(Z)$-module, and hence a structure of $\EE$-vector space of
  dimension $2\dim Z/[\EE:\qq]$. In particular, $V_\qq\otimes_\qq\cc$
  is a free $\EE\otimes_\qq \cc$ module of rank $2\dim Z/[\EE:\qq]$. That
  is,
  \[ H_1(Z, \cc)=H_1(Z,\qq)\otimes_\qq\cc = \bigoplus_{\sigma\in
    \Sigma_\EE} H_1(Z, \cc)_\sigma, \] where each $H_1(Z, \cc)_\sigma$
  is a complex vector space of dimension $2\dim Z/[\EE:\qq]$.  On the
  other hand, we have the Hodge decomposition \cite[Chapter
  1]{Mumford_AV},
  \[ H_1(Z, \cc)= H^{-1,0}(Z)\oplus H^{0, -1}(Z),\] where
  $H^{-1,0}(Z)$ and $H^{0, -1}(Z)$ are mutually complex conjugate
  $\cc$-vector spaces of dimension $\dim(Z)$. The splitting is
  $\End^0(Z)$-invariant and the $\End^0(Z)$-module $H^{-1,0}(Z)$ is
  canonically isomorphic to $\Lie_\cc(Z)$. For any $\sigma\in
  \Sigma_\EE$, we write $\bar{\sigma}$ for the composition of
  $\EE\xrightarrow{\sigma} \cc$ with the complex conjugation map
  $\cc\to \cc$. Then
\[ H_1(Z, \cc)_\sigma\cong  \Lie_\cc(Z)_\sigma \oplus \overline{\Lie_\cc(Z)_{\bar{\sigma}}}.\]
Therefore,
\begin{equation}
  \label{eq:hodge-decomposition}
  g(\sigma)+g(\bar{\sigma})= 2\dim Z/[\EE:\qq].
\end{equation}
\end{sect}

\begin{sect}
  Let $\End^0(X, \ii)$ be the centralizer of $\ii(E)$ in
  $\End^0(X)$. As $\End^0(X)$ itself is a semisimple $\qq$-algebra,
  $\End^0(X, \ii)$ is a semisimple $E$-algebra. Let $m=2\dim
  X/\varphi(N)$. We have
 \[  E \subseteq  \End^0(X, \ii) \subseteq \End_E(H_1(X, \qq))\cong
 \Mat_m(E).\] 
Suppose that $\dim X= \varphi(N)=[E:\qq]$, then we have the following
possibilities for $\End^0(X, \ii)$:
\begin{equation}
  \label{eq:28}
 \End^0(X, \ii)=
 \begin{cases}
   E, \\
   L, \\
   E\oplus E,\\
   \Mat_2(E),\\
 \end{cases}
\end{equation}
where $L/E$ is a field extension of degree 2. In the last three cases,
$X$ is an abelian variety of CM-type, as observed in
\cite[Theorem 3.1]{MR2349666}. We claim that $\End^0(X,
\ii)\neq \Mat_2(E)$ if there exists $a\in \umod{N}$ with
$h(a)=1$. Indeed, as in
Subsection~\ref{subsec:def-generalized-mul-type},\[\Lie_\cc(X)=\bigoplus_{a\in
  \umod{N}}\Lie_\cc(X)_a,\] and each $\Lie(X)_a$ is a $\End^0(X,
\ii)$-invariant complex vector space of dimension $h(a)$. On the other
hand, $\Mat_2(E)\otimes_\qq \cc\cong \oplus_{a\in \umod{N}}
\Mat_2(\cc)$, and a minimal module of $\Mat_2(\cc)$ is 2-dimensional.
\end{sect}
\begin{lem}\label{lem:centralizer-is-not-matrix-algebra}
Suppose that $N\not \in \{3, 4, 6, 10\}$, and $f(x)\in \cc[x]$ is
polynomial of degree $3$ with no multiple roots. Then
$\End^0(J_{f,N}^\nw, \ii)\neq \Mat_2(E)$. 
\end{lem}
\begin{proof}
  The multiplication type function $h$ is given by (\ref{eq:34}). It
  takes value $1$ for some $a\in \umod{N}$ by
  Proposition~\ref{prop:exist-value-one} if $N\not\in \{3, 4, 6,
  10\}$. 
\end{proof}

\begin{rem}
  Let $\lambda: X\to X^\vee$ be a polarization on $X$ that induces a
  Rosati involution $\alpha \mapsto \alpha^\dagger$ on
  $\End^0(X)$. Suppose that $E$ is invariant under the Rosati
  involution (i.e., $E^\dagger =E$), then its centralizer $\End^0(X,
  \ii)$ is also invariant under the Rosati involution. In particular,
  if $\End^0(X, \ii)=L$, then $L$ is a CM-field. This holds if
  $X=J_{f,N}^\nw$ and we take $\lambda$ to be the restriction of the
  canonical principal polarization of $J(C_{f,N})$ to $J_{f,N}^\nw$. 
\end{rem}

\begin{sect}\label{subsec:multiplication-type-primitive}
  For $s\in \umod{N}$, we write $\theta_s: \umod{N}\to \umod{N}$ for
  the multiplication by $s$ map: $a\mapsto sa$. The generalized
  multiplication type $h:\umod{N}\to \nn$ is said to be
  \textit{primitive} if $h\circ \theta_s = h \Leftrightarrow s=1$.
  Suppose that $\dim X= \varphi(N)(n-1)/2$, and $h$ is given by
  (\ref{eq:34}). Then $h$ is primitive if one of the following
  condition holds: 
  \begin{itemize}
  \item $\gcd(n, N)=1$,
    (by Proposition~\ref{prop:arithmetic-results-rigidity-of-h});
  \item $n=3$, and $N=3^r\geq 9$, (by \cite[Lemma 4.2]{xue_JNT}).
  \end{itemize}
\end{sect}

For the rest of the section, we will do a case by case study of the
first three cases of (\ref{eq:28}). The case when $\End^0(X, \ii)=E$
and when $N=q=p^r$ is a prime power was treated in \cite{MR2349666} for
$n=\deg f=3,4$, and in \cite{MR2166091} and \cite{MR2471095} for
$n\geq 5$, where it was assumed that $\gcd(q,n)=\gcd(p,n)=1$. The case
when $q=p^r$ and $p\mid n$ was treated in \cite{xue_JNT}. We will
extend these results to a more general $N$.

First, we state the following theorem of Zarhin \cite[Theorem
2.3]{MR2040573}.

\begin{thm} \label{thm:Zarhin_multiplication_type} Let the notation be
  the same as in Subsection~\ref{subsec:Hodge-decomposition}.  Suppose
  that $\EE$ (identified with its image via $\jj$) contains the center
  $\CV_Z$ of $\End^0(Z)$, and $\EE/\CV_Z$ is Galois, then
\[ g(\sigma \circ \kappa)=
  g(\sigma), \qquad  \forall \sigma\in \Sigma_\EE, \forall \kappa \in \Gal(\EE/\CV_Z). \]
\end{thm}
The statement of \cite[Theorem 2.3]{MR2040573} is restricted to the
case that $\EE/\qq$ is Galois. However, its proof shows that the
theorem holds as long as $\EE$ is Galois over $\CV_Z$.  The next
proposition generalizes \cite[Theorem 4.2]{MR2349666} and
\cite[Corollary 2.2]{MR2166091} and follows the main idea of their
proofs. 


\begin{prop}\label{prop:centeralizer-is-E}
  If $\End^0(X, \ii)=E$, and the generalized multiplication type $h:
  \umod{N}\to \nn$ of $(X, \ii)$ is primitive, then $X$ is absolutely
  simple, and $\End^0(X)\cong E$.
\end{prop}

\begin{proof}
  Since the centralizer of $E$ in $\End^0(X)$ coincides with $E$, the
  center $\CV_X$ of $\End^0(X)$ is contained in $E$. If $E\neq \CV_X$,
  by Theorem~\ref{thm:Zarhin_multiplication_type}, there exists $s\in
  \umod{N}$, $s\neq 1$ such that $h(a)= h(sa)$ for all $a\in
  \umod{N}$, which is not the case by our assumption.  Therefore, $E$
  coincides with the center of $\End^0(X)$. Hence
  $\End^0(X)= \End^0(X, \ii)$, which equals to $E$ by assumption.
\end{proof}


\begin{prop}\label{prop:centralizer-is-quadratic-extension}
  Suppose $\dim X=\varphi(N)$. If $\End^0(X, \ii)=L$, a quadratic
  extension of $E$, and the generalized multiplication type $h$ of
  $(X, \ii)$ is primitive, then $X$ is absolutely simple, and
  $\End^0(X)\cong L$.
\end{prop}

\begin{proof}
As $X$ is an abelian variety of CM-type, and by
\cite[Section II.5]{Shimura_CM}, we see that

\begin{enumerate}[(i)]
\item $X$ is isogenous to a product $Y\times \cdots \times Y$
  with a simple abelian variety $Y$. 
\item $F:=\End^0(Y)$ is a CM subfield of $L$ with $[F:\Q]=2\dim Y$.
\item $\End^0(X)\cong \Mat_t(F)$, where $t=[L:F]$ is the number of
  factors of $Y$ in the product $Y\times \cdots \times Y$. 
\item The center $\CV_X$ of $\End^0(X)$ coincides with $F$. 
\end{enumerate}

Note that we have a tower of fields $E\subseteq EF\subseteq L$. Since
$[L:E]=2$, either $EF=E$ or $EF=L$. We claim that $EF=L$. Suppose
otherwise, then $\CV_X=F\subseteq E$. Same argument as in the proof of
Proposition~\ref{prop:centeralizer-is-E} shows that $E=F$ and
    \[\End^0(X,\ii)=\End^0(X)\cong \Mat_t(F)\neq L. \]
    Contradiction to our assumption. Therefore, $EF=L$. If $L=F$, then
    $t=1$, and $X=Y$ is simple. Furthermore,
    \[ \End^0(X)=F=L=\End^0(X,\ii).\] So for the rest of the proof, we
    assume that $EF=L$, $F\neq L$ and show that this leads to a
    contradiction.

    Since $E=\qq(\zeta_N)$ is Galois over $\qq$, $L=EF$ is Galois over
    $F$ with $\Gal(L/F)\leq \Gal(E/\qq)$. Let $F_0:=F\cap E$. Then
    $[F:F_0]= [L:E]=2$, and $F/F_0$ is Galois as well.  By \cite[
    Theorem VI.1.14]{MR1878556} , $L/F_0$ is Galois, with
    $\Gal(L/F_0)= \Gal(L/F)\times \Gal(L/E)$. We write $\iota$ for the
    unique generator of $\Gal(L/E)$. It commutes with all elements of
    $\Gal(L/F)$.

    Let $g: \Sigma_L\to \nn$ the generalized multiplication type of
    $(X, L\hookrightarrow \End^0(X))$, and $h_0: \Sigma_F\to \nn$ be
    the CM-type of $(Y, F\hookrightarrow \End^0(Y))$. By
    (\ref{eq:hodge-decomposition}), both $g$ and $h_0$ takes values
    $0$ and $1$ only. Since $\Lie_\cc(X)$ is isomorphic to the direct
    sum of $t$-copies of $\Lie_\cc(Y)$, $g$ is induced from $h_0$ in
    the following sense:
    \[ g(\sigma)=1 \Leftrightarrow h_0(\sigma\mid_F)= 1, \qquad
    \forall \sigma\in \Sigma_L.\] In particular, $g(\sigma \kappa)=
    g(\sigma)$ for all $\kappa \in \Gal(L/F)$. On the other hand,
\[ h(\sigma\mid_E)= g(\sigma)+ g(\sigma\iota), \qquad \forall
\sigma\in \Sigma_L.\] 
It follows that for any $\kappa \in \Gal(L/F)$, $\sigma\in \Sigma_L$,  
\[
\begin{split}
h((\sigma\mid_E) \circ (\kappa\mid_E))&= h(\sigma\kappa \mid_E)=
g(\sigma\kappa)+ g(\sigma\kappa \iota)= g(\sigma\kappa)+g(\sigma \iota
\kappa)\\&= g(\sigma)+g(\sigma\iota)= h(\sigma\mid_E).
\end{split}.\]
This again contradicts the assumption that $h$ is primitive. 
\end{proof}

\begin{cor}\label{cor:nw-factor-of-jacobian-quad-ext}
  Suppose that $f(x)\in \cc[x]$ is a polynomial of degree $3$ with no
  multiple roots, $3\nmid N$ and $N\not\in \{4, 10\}$.  Suppose
  further that $\End^0(J_{f,N}^\nw)\supseteq L \supseteq E$, where $L$
  is a quadratic field extension of $E$. Then $\End^0(J_{f,N}^\nw)=L$.
\end{cor}
\begin{proof}
  Clearly, $\End^0(J_{f,N}^\nw, \ii) \supseteq L$. By
  Lemma~\ref{lem:centralizer-is-not-matrix-algebra},
  $\End^0(J_{f,N}^\nw, \ii)\neq \Mat_2(E)$ since $N\neq 4, 10$. It
  follows from (\ref{eq:28}) that $\End^0(J_{f,N}^\nw, \ii)= L$. We
  have mentioned in Subsection~\ref{subsec:multiplication-type-primitive}
  that the multiplication type of $(J_{f,N}^\nw, \ii)$ is primitive,
  so the corollary follows from
  Proposition~\ref{prop:centralizer-is-quadratic-extension}.
\end{proof}

We also give the proof of Theorem~\ref{thm:galois-group-is-s3}. 

\begin{proof}[Proof of Theorem~\ref{thm:galois-group-is-s3}]
  It is enough to show that $\End^0(J_{f,q}^\nw, \ii)$ is simple.  The
  group $\mathbf{S}_3$ is doubly transitive and $k$ is assumed to
  contain $\qq(\zeta_q)$. This allows us to first apply \cite[Theorem
  5.13]{MR2349666}, then \cite[Lemma 3.8]{MR2471095}, and in the end,
  combining the proof of \cite[Theorem 3.12]{MR2471095} together with
  Proposition~\ref{prop:exist-value-one} to get the desired result.
\end{proof}

\begin{sect}\label{subsec:case-centralizer-is-E-plus-E}
  Suppose that $\dim X= \varphi(N)$ and $\End^0(X, \ii)=E\oplus
  E$. Then $X$ is isogenous to $Y_1\times Y_2$, and each $Y_i$ is an
  abelian variety of dimension $\varphi(N)/2$ with complex
  multiplication by $E$.  Let $p_i: E\oplus E\to E$ be the projection
  onto $i$-th factor for $i=1,2$, and $\jj_i:
  E\hookrightarrow \End^0(Y_i)$ be the composition of
  $E\xrightarrow{\ii} \End^0(X, \ii) \xrightarrow{p_i} E
  \subseteq \End^0(Y_i)$. If we write $g_i$ for the CM-type of $(Y_i,
  \jj_i)$, then $h=g_1+g_2$. By the criterion of Shimura-Taniyama
  \cite{Shimura_CM},
  \begin{itemize}
  \item $Y_1\sim Y_2$ if and only if $\exists s\in \umod{N}$ such that
    $g_1\circ \theta_s= g_2$. 
  \item $Y_i$ is simple if and only if $g_i$ is primitive. 
  \end{itemize}
  When both $Y_1$ and $Y_2$ are simple, $\End^0(Y_1)=\End^0(Y_2)=E$.
  If $Y_1\not\sim Y_2$, then
  $\End^0(X)= \End^0(Y_1)\oplus \End^0(Y_2)=E\oplus E$; otherwise
  $Y_1\sim Y_2$, and we have $\End^0(X)= \Mat_2(E)$. On the other
  hand, say $Y_1$ is not simple, then the group $\{s\in \umod{N}\mid
  g_1\circ \theta_s= g_1\}$ is nontrivial. Let $t$ be the order of
  this group, and $F$ be the its fixed subfield in
  $\qq(\zeta_N)$. Then $Y_1\sim Z^t$, where $Z$ is a simple complex
  abelian variety with complex multiplication by $F$. In particular,
  $\End^0(Y_1)=\Mat_t(F)$.
\end{sect}

  \begin{sect}\label{subsec:notation-arithmetic-result-on-cm-type}
    Since $g_i$ only takes value in $\{0,1\}$, $h(a)=0$ if and only if
    both $g_1(a)$ and $g_2(a)$ are zero. Note that $h(a)=\fl{3a/N}$
    takes value 0 for all $1\leq a< N/3$ and $\gcd(a,N)=1$, so the
    same holds for both $g_1(a)$ and $g_2(a)$.  Let $\TZ_N$ be the set
    of functions
  \begin{equation}
    \label{eq:35}
    \TZ_N=\{g: \umod{N}\to \{0, 1\}\mid g(a)+g(-a)=1,  g(a)=0 \text{ if }
    1\leq a <N/3\}.
  \end{equation}
  Suppose that $N=q=p^r$ is a prime power. We are interested in the
  set 
  \begin{equation}
    \label{eq:37}
 \SZ_q:=\{s\in \umod{q} \mid s\neq 1, \text{and } \exists g\in
  \TZ_q \text{ such that } g\circ \theta_s\in \TZ_q\}.     
  \end{equation}
Clearly, $s\in \SZ_q$ if and only if $s^{-1}\in \SZ_q$. 
  In Section~\ref{sec:arithmetic-results}, it will be shown
  \begin{equation}
    \label{eq:36}
    \SZ_q=
    \begin{cases}
      \{2, (q+1)/2\}  \qquad & \text{ if } p \geq 5;\\
      \{2, (q+1)/2, q/3-1, 2q/3-1\} \qquad & \text{ if } p=3 \text{
        and } q\geq 9;\\
      \{q/2-1\} \qquad & \text{ if } p=2 \text{ and } q\geq 16. 
    \end{cases}
  \end{equation}
  Moreover, for each $s\in \SZ_q$, there exists a unique $g\in \TZ_q$
  such that $g\circ \theta_s\in \TZ_q$. It follows that other than
  some exceptional cases, $Y_i$ are all simple. 
  \end{sect}

\begin{prop}\label{prop:simple-cm-by-cyclotomic}
  Let $q=p^r$ be a prime power with $p$ odd. Assume that $q\neq 7$,
  and $q\geq 9$ if $p=3$. Let $(Y,\jj)$ be a pair consisting a complex
  abelian variety of dimension $\varphi(q)/2$ and an embedding $\jj:
  \qq(\zeta_q)\hookrightarrow \End^0(Y)$. Suppose that the CM-type of
  $(Y, \jj)$ is given by a function $g: \umod{q}\to \{0, 1\}$ such
  that $g(a)=0 $ for all $1\leq a < q/3$.  Then $Y$ is simple with
  $\End^0(Y)\cong\qq(\zeta_q)$.
\end{prop}
\begin{proof}
  We need to show that $g$ is primitive. If $p\neq 3$ and
  $q\neq 7$, this is shown in
  Corollary~\ref{cor:primitive_cm_type}. If $p=3$ and $q\geq 9$, this
  is shown in Corollary~\ref{cor:primitive-cm-type-p-is-3}.
\end{proof}

\begin{prop}\label{prop:centralizer-is-e-plus-e}
  Assume that $\dim X= \varphi(N)$, the multiplication type $h$ of
  $(X, \ii)$ is given by (\ref{eq:34}) with $n=3$, and $\End^0(X,
  \ii)=E\oplus E$. Assume further that $N=q=p^r$ is a prime power,
  $q>3$ if $p=3$, and $q>4$ if $p=2$. Then
  \begin{itemize}
  \item if $p\geq 5$ and $q\neq 5, 7$, then $\End^0(X)= E\oplus E$;
  \item if $q=5, 9$, then $\End^0(X)= \Mat_2(E)$; 
  \item if $q=7$, then $\End^0(X)= \Mat_3(\qq(\sqrt{-7}))\oplus
    \qq(\zeta_7)$. 
  \item if $q=3^r\geq 27$, then $\End^0(X)$ is either $E\oplus E$ or
    $\Mat_2(E)$. 
  \item If $q=8$, then $\End^0(X)=\Mat_2(\qq(\sqrt{-1}))\oplus
    \Mat_2(\qq(\sqrt{-2}))$. 
  \item If $q\geq 16$, then $\End^0(X)=E\oplus E$ or $E\oplus
    \Mat_2(\qq(\alpha))$, where $\alpha= 2\sqrt{-1}\sin(2\pi/q)$.  
  \end{itemize}
\end{prop}
\begin{proof}
  First suppose that $p\geq 5$. By
  Lemma~\ref{lem:splitting_for_2_twist}, $Y_1\not\sim Y_2$ if $q\neq
  5$.  By Proposition~\ref{prop:simple-cm-by-cyclotomic}, both $Y_1$
  and $Y_2$ are simple if $q\neq 7$. It follows that $\End^0(X)=
  E\oplus E$ when $q\neq 5, 7$. If $q=5$, then $g_1$ and $g_2$ are
  uniquely determined (up to relabeling) by $h$, and $g_1=g_2\circ
  \theta_2$ by Remark~\ref{rem:case-q-is-5}. So $Y_1\sim Y_2$, and
  $\End^0(X)=\Mat_2(\qq(\zeta_5))$. Similarly, if $q=7$, by
  Remark~\ref{rem:case-q-is-7}, $g_1$ and $g_2$ are uniquely
  determined by $h$ up to relabeling. One checks that $g_2$ is
  primitive, hence $\End^0(Y_2)=\qq(\zeta_7)$; and
  \[ g_1\circ \theta_s= g_1 \Leftrightarrow s\in \dangle{2} \leq
  \umod{7}.\] So the CM-type $g_1$ is induced from $\qq(\sqrt{-7})$,
  the fixed subfield of $\qq(\zeta_7)$ by $\dangle{2}\leq
  \Gal(\qq(\zeta_7)/\qq)$.  Therefore, $Y_1\sim Z^3$, where $Z$ is an
  elliptic curve with complex multiplication by $\qq(\sqrt{-7})$, and
  $\End^0(Y_1)= \Mat_3(\qq(\sqrt{-7}))$.

  If $p=3$ and $q=3^r\geq 9$, then by
  Proposition~\ref{prop:simple-cm-by-cyclotomic}, both $Y_1$ and $Y_2$
  are simple and $\End^0(Y_i)=E$. If $q=9$, then by
  Remark~\ref{rem:isogeny-q-is-9}, there is a unique way (up to
  labeling) to write $h=g_1+g_2$ , and $g_1=g_2\circ
  \theta_2$. Therefore, $Y_1\sim Y_2$, and $\End^0(X)=\Mat_2(E)$. If
  $q\geq 27$, $\End^0(X)$ depends on the specific form of
  $g_i$. Suppose that $g_1$ is of the form given by
  (\ref{eq:29}). Then $g_2= g_1\circ \theta_s$ with $s=q/3-1$. So
  $Y_1\sim Y_2$ and $\End^0(X)= \Mat_2(E)$. Otherwise, $Y_1\not\sim
  Y_2$ by Lemma~\ref{lem:not-isogenous-q-power-of-3}, and
  $\End^0(X)=E\oplus E$.
 
  If $q=8$, once again, $g_1$ and $g_2$ are uniquely determine up to
  labeling. By Remark~\ref{rem:case-q-is-8}, $g_1\circ \theta_5=g_1$,
  and $g_2= g_2\circ\theta_3$.  The fixed subfield of
  $\qq(\zeta_8)=\qq(\sqrt{-1}, \sqrt{-2})$ by $\dangle{5}\leq
  \umod{8}$ is $\qq(\sqrt{-1})$, so $\End^0(Y_1)=
  \Mat_2(\qq(\sqrt{-1}))$ and $Y_1$ is isogenous to a square of an
  elliptic curve with complex multiplication by
  $\qq(\sqrt{-1})$. Similarly, $Y_2$ is isogenous to the square of an
  elliptic curves with complex multiplication by
  $\qq(\sqrt{-2})$. Therefore, $\End^0(X)= \Mat_2(\qq(\sqrt{-1}))\oplus
    \Mat_2(\qq(\sqrt{-2}))$.

 If $q=2^r\geq 16$, then $\End^0(X)$ depends on the
 specific form of $g_i$'s. If $g_1$ is of the form given in
 (\ref{eq:30}), by Lemma~\ref{lem:q-is-a-power-of-2}, 
 \[ g_1\circ \theta_s= g_1 \Leftrightarrow s=1, 2^{r-1}-1.\] One
 easily checks that the fixed subfield of $\qq(\zeta_q)$ by
 $\dangle{2^{r-1}-1}$ is $\qq(\alpha)$. By
 Lemma~\ref{lem:q-is-a-power-of-2} again, $g_2=h-g_1$ is primitive. So
 $\End^0(X) = \Mat_2(\qq(\alpha))\oplus E$. If neither $g_1$ nor $g_2$
 is of the form in (\ref{eq:30}), then both $g_i$ are primitive, and
 $Y_1\not\sim Y_2$ by
 Lemma~\ref{lem:not-isogenous-q-is-power-of-2}. Therefore,
 $\End^0(X)=E\oplus E$.
\end{proof}


\begin{sect}
  Recall that $C_{\lambda, q}$ denotes the curve
  \[ y^q=x(x-1)(x-\lambda),\] where $\lambda$ lies on the punctured
  complex plane with the points $0$ and $1$ removed.  Fix $\lambda$
  such that $J_{q,\lambda}^\nw$ is not of CM-type (Such a $\lambda$
  exists if $q\neq 4$). We may construct a
  Shimura datum $(G, X)$ from $J_{q,\lambda}^\nw$ in the following
  way. Let $V$ be the $\qq$-vector space $H_1(J_{q,\lambda}^\nw,
  \qq)$. It carries a natural structure of $\qq(\zeta_q)$-vector
  spaces of dimension $2$. The canonical principal polarization on
  $J(C_{\lambda,q})$ induces on $V\subseteq H_1(C_{\lambda, q}, \qq)$
  a nondegenerate alternating $\qq$-bilinear form $\psi$ which
  satisfies the condition
  \[ \psi(eu,v)= \psi(u, \bar{e}v), \qquad \forall e\in \qq(\zeta_q),
  \forall u, v\in V. \] Let $\CSp(V_\qq, \psi)$ be the group of
  symplectic similitudes of $\psi$, and \[ G=\GL_{\qq(\zeta_q)}(V)\cap
  \CSp(V, \psi).  \] Let $\mathbb{S}:=\Res_{\cc/\rr}\mathbb{G}_m$ be
  the Deligne torus, and $h_0: \mathbb{S}\to G\otimes_\qq\rr$ the
  homomorphism of $\rr$-algebraic groups that defines the Hodge
  structure on $V$. We set $X$ to be the $G(\rr)$-conjugacy class of
  $h_0$. Let $\mathbb{A}^f$ be the finite adeles of $\qq$, and $K$ a
  compact open subgroup of $G(\mathbb{A}^f)$.  By the moduli
  interpretation of Shimura varieties of PEL-type (\cite[Scholie
  4.11]{MR0498581}), the classifications of the endomorphism algebra
  in Theorem~\ref{thm:main} holds for any abelian variety $A$ (with
  additional structure) corresponding to a complex point on the
  Shimura variety
  \[ \mathrm{Sh}_K(G, X):= G(\qq)\backslash X\times
  G(\mathbb{A}^f)/K.\] Indeed, \cite[Scholie 4.11(a)]{MR0498581} shows
  that the general multiplication type of $\qq(\zeta_q)$ on
  $\Lie_\cc(A)$ coincides with that of $J_{\lambda,q}^\nw$ for all
  such $A$, and Theorem~\ref{thm:main} was obtained purely by studying
  the general multiplication types.  For the same reason, we may
  replace $G$ by the group $G_1$ defined in \cite[Scholie
  4.13]{MR0498581} and obtain a similar result. 
\end{sect}


\section{Automorphisms and construction of examples}\label{sec:autom-constr-exampl}
Throughout this section, we assume that $k$ is an algebraically closed
field of characteristic zero, $f(x)\in k[x]$ is a monic polynomial of
degree 3 with no multiple roots.

\begin{sect}
  Suppose that $\gcd(N,3)=1$, and
  $f(x)=\prod_{i=1}^3(x-\alpha_i)$. The set of fixed points of
  $\delta_N$ on $C_{f,N}(\bar{k})$ is
\[ \SV:=\{ P_i:=(\alpha_i, 0)\}_{i=1}^3\cup \{\infty\},\]
where $\infty$ is the unique point at infinite for $C_{f,N}$. Let
$\ddiv$ denote the divisor of a function. Then
\begin{gather}
  \ddiv y =  P_1+P_2+P_3-3\infty,\\
  \ddiv (x-\alpha_i)= N P_i- N\infty.
\end{gather}  
Choose $s, t\in \zz$ such that $3s+Nt=1$. We have
  \begin{equation}
    \label{eq:31}
    \begin{split}
    \ddiv y^s(x-\alpha_1)^t& = s(P_1+P_2+P_3-3\infty)+ t
    (NP_1-N\infty)\\
    &= s(P_1+P_2+P_3)+tNP_1 - \infty      .
    \end{split}
 \end{equation}

\end{sect}

\begin{sect}
  Let $\Aut(C_{f,N})$ be the absolute automorphism group of
  $C_{f,N}$. We write $H$ for the normalizer of $\dangle{\delta_N}$ in
  $\Aut(C_{f,N})$. Suppose that $H\neq \dangle{\delta_N}$. We consider
  an element $\phi$ in $H$ but not in $\dangle{\delta_N}$.  Then
  $\phi$ permutes elements of $\SV$.  We claim that $\phi \infty=
  \infty$ if $N\neq 2, 4$. Otherwise, say $\phi^{-1} \infty = P_1$,
  then $P_1\neq \phi^{-1} P_i$ for all $1\leq i\leq 3$.  Without lose
  of generality, we assume that $\phi^{-1} P_2\neq \infty$. Note that
  \begin{equation}
    \label{eq:32}
 \ddiv \phi^*y= \phi^{-1} P_1 + \phi^{-1} P_2 + \phi^{-1} P_3 - 3P_1.    
  \end{equation}
  Using (\ref{eq:31}) to replace $\infty$ on the right hand side of
  (\ref{eq:32}), we get an divisor supported on $\{P_1, P_2, P_3\}$
  that's linear equivalent to zero. By \cite[Lemma 2.7]{xue_JNT}, a
  divisor of degree 0 supported on $\{P_1, P_2, P_3\}$ is linear
  equivalent to zero if and only if all coefficients of the $P_i$'s
  are congruent to each other modulo $N$. (A priori, \cite[Lemma
  2.7]{xue_JNT} only proved the statement for the case $N$ is a prime
  power, however, the same argument applies for any arbitrary $N$.)
  Comparing the coefficient of $\phi^{-1} P_2$ and $P_1$, we see that
  $1\equiv -3 \pmod{N}$, which contradicts the assumption that $N\neq
  2,4$.
\end{sect}

\begin{sect}
  Suppose that $\gcd(N,3)=1$ and $N\nmid 4$. Let $k(C_{f,N})$ be the
  field of rational functions of $C_{f,N}$.  The fixed subfield of
  $k(C_{f,N})$ by $\dangle{\delta_N}$ is $k(x)$, and every element of
  $H$ sends $k(x)$ to itself, therefore, the action of $H$ on $k(x)$
  induces an embedding $H/\dangle{\delta_q}\subseteq \Aut(k(x)/k)$.  It
  is well known (cf. \cite[Corollary 6.65]{MR2674831}) that
  $\Aut(k(x)/k)$ is isomorphic to the group of all linear fractional 
  transformations over $k$. Since a linear fractional transformation
  is uniquely determined by its image on any three distinct points, we
  see that $H/\dangle{\delta_q}\hookrightarrow \Perm{\{P_1, P_2, P_3\}}\cong
  \mathbf{S}_3$.

  Let $\phi_x\in H/\dangle{\delta_q}$ be the automorphism of $k(x)$
  induced by $\phi^*: k(C_{f,N})\to k(C_{f,N})$. Since $\phi\not\in
  \dangle{\delta_N}$, $\phi_x$ is nontrivial, so the order of $\phi_x$
  is either $2$ or $3$.  Since $\phi(\infty) = \infty$,
  $\phi_x(x)=tx+b$, for some $t, b\in k$. Therefore, 
\[ \phi_x^2(x)=t^2x+tb+b, \qquad \phi_x^3(x)= t^3x+(1+t+t^2)b.\]
If $\phi_x$ has order $2$, then $t=-1$, and if $\phi_x$ has order $3$,
$t = \omega$, where $\omega$ is a primitive 3rd root of unity. By
changing the $x$ coordinate appropriately, we may assume that
$\phi_x(x)=-x$ or $\phi_x(x)=\omega x$ respectively. More explicitly,
if $\ord\phi_x=2$, we replace $x$ by $x-b/2$, and if $\ord\phi_x=3$,
we replace $x$ by $x-b(1-\omega^2)/3$.

Now since $\phi$ permutes the points $P_1, P_2, P_3$ and fixes
$\infty$,

\[
\begin{split}
\ddiv \phi^* f(x) &= \phi^{-1} (\ddiv f(x))= \phi^{-1} (\ddiv y^N)= \phi^{-1} N( P_1+P_2+P_3 -3\infty)\\&=   N( P_1+P_2+P_3 -3\infty )= \ddiv f(x).   
\end{split}
 \]

 In particular, $\phi^* f(x)= c f(x)$ for some $c\in k^*$. On the
 other hand, $\phi_x(f(x))= f(\phi_x x)$. Comparing the coefficients,
 we see that $f(x)=x^3+B_0x$ if $\phi_x$ has order $2$, and $f(x)=
 x^3+C_0$ if $\phi_x$ has order $3$. In both cases, the coefficient of
 $x^2$ is zero.  We note that for $f(x)=x^3+A_0x^2+B_0x+C_0$, there
 exists a unique $b\in k$ such that the coefficient of $x^2$ in
 $f(x-b)$ is zero. Indeed, $b=A_0/3$. It follows that
 $H/\dangle{\delta_N}$ does not contain both an element of order $2$
 and an element of order $3$. In other words, $H/\dangle{\delta_N}\neq
 \mathbf{S}_3$.

If $\phi_x$ has order 2, $\phi^*(f(x))= -f(x)$. It follows that
$\phi^*(y^N)= - y^N$. If $N$ is odd, then $\phi^*(y)=
-\zeta_N^iy$ for some $i\in \zmod{N}$.  If $N$ is even, then
$\phi^*(y) = (\zeta_{2N})^{2i+1}y$ for some $i \in \zmod{N}$. If
$\phi_x$ has order $3$, then $\phi^*(f(x))=f(x)$, so
$\phi^*(y^N)= y^N$. Therefore, $\phi^*(y)= \zeta_N^i y$
for some $i\in \zmod{N}$.
\end{sect}

\begin{thm}\label{thm:normalizer-automorphism}
  Let $k$ be an algebraically closed field of characteristic zero, and
  $f(x)\in k[x]$ a monic polynomial of degree 3 with no multiple
  roots. After a unique change of variable of the form $x\mapsto x-b$
  for some $b\in k$, we may and will assume that $f(x)=x^3+B_0x+C_0$.
  Suppose that $(N, 3)=1$ and $N\neq 2, 4$. Let $H\leq \Aut(C_{f,N})$
  be the normalizer of $\dangle{\delta_N}$. We have the following
  cases: 
  \begin{itemize}
  \item if $B_0C_0\neq 0$, then $H=\dangle{\delta_N}\cong \zmod{N}$;
  \item if $B_0=0$ and $C_0\neq 0$, then $H\cong \zmod{3N}$, and $H$
    is generated by the automorphism $(x,y)\mapsto (\omega x,
    \zeta_Ny)$;
  \item if $B_0\neq 0$ and $C_0=0$, then $H\cong \zmod{2N}$, and $H$
    is generated by the automorphism $(x,y)\mapsto (-x,
    \zeta_{2N}y)$. 
  \end{itemize}
\end{thm}


\begin{sect}\label{subsec:lambda}
  Suppose that $k=\cc$, and $f(x)=x(x-1)(x-\lambda)=
  x^3-(1+\lambda)x^2+\lambda x$. We take $b=-(1+\lambda)/3$, then
\[f\left(x+\frac{1+\lambda}{3}\right)=x^3+\left(\frac{\lambda-1-\lambda^2}{3}\right)x
-\frac{(1+\lambda)(\lambda-2)(2\lambda-1)}{27}.
\]
If it is of the form $x^3+C_0$, then $\lambda= (1\pm \sqrt{-3})/2$; if
it is of the form $x^3+B_0x$, then $\lambda= -1, 2, 1/2$.   
\end{sect}

\begin{proof}[Proof of Corollary~\ref{cor:automorphism}]
It was remarked in
Subsection~\ref{subsec:general-results-about-curves} that 
\[ \Aut(C_{f,q})\subseteq \Aut(J(C_{f,q}))
\subseteq \End(J(C_{f,q})). \] By Theorem~\ref{thm:main},
$\End^0(J(C_{f,q}))$ is commutative if $p>7$, and hence $\Aut(C_{f,q})$
coincides with the normalizer of $\dangle{\delta_q}$. So
Corollary~\ref{cor:automorphism} follows directly from
Theorem~\ref{thm:normalizer-automorphism} and
Subsection~\ref{subsec:lambda}. 
\end{proof}

For the rest of this section, we will try to construct examples of
$J_{f,N}^\nw$ whose endomorphism algebras take the form as predicted
by Proposition~\ref{prop:centralizer-is-quadratic-extension} or
Proposition~\ref{prop:centralizer-is-e-plus-e}. Our  method is to
use curves with extra automorphisms. 

 \begin{sect}
   Consider the curve $C_{f_1, N}$ with $f_1(x)=x^3-x$, $3\nmid N$,
   $N\in 2\zz$ even, and $N\not\in \{4, 10\}$. Let $\gamma_{2N}\in
   \Aut(C_{f_1,N})$ be the automorphism defined by $(x,y)\mapsto (-x,
   \zeta_{2N}y)$. Clearly, $\gamma_{2N}^2= \delta_N$. So
   $J_{f_1,N}^\nw$ is $\gamma_{2N}$-invariant, and we have an
   embedding
   \[ \qq(\zeta_{2N})\cong
   \qq(\zeta_N)[T]/(T^2-\zeta_N)\hookrightarrow \End^0(J_{f_1,
     N}^\nw), \qquad T\mapsto \gamma_{2N}\mid_{ J_{f_1, N}^\nw}.\] It
   follows from Corollary~\ref{cor:nw-factor-of-jacobian-quad-ext}
   that $J_{f_1,N}^\nw$ is absolutely simple with
   $\End^0(J_{f_1,N}^\nw)=\qq(\zeta_{2N})$.
 \end{sect}
 \begin{sect}
   Let $f_1$ be as above, and $N=q=p^r$ be a prime power with $p$
   odd. If $p=3$, we assume that $q\geq 9$. Let $\gamma_2\in
   \Aut(C_{f_1, q})$ be the automorphism defined by $(x,y)\mapsto (-x,
   -y)$, then $\gamma_2$ commutes with $\delta_N$, so $J_{f_1, q}^\nw$
   is $\gamma_2$ invariant. With an abuse of notation, we still write
   $\gamma_2$ for the restriction $\gamma_2\mid_{J_{f_1,
       q}^\nw}\in \End(J_{f_1, q}^\nw)$. Clearly, $\gamma_2^2=\Id$. We
   claim that $\gamma_2\neq \pm 1$.  Let $b= (q-1)/2$, and $c=
   (q+1)/2$, then $\gcd(b,q)=\gcd(c,q)=1$.  By
   (\ref{eq:basis-of-diff-first-kind}), both $dx/y^b$ and $dx/y^c$ are
   differentials of first kind on $C_{f_1, q}$. Clearly they are
   eigenvectors 
   corresponding to distinct eigenvalues for
   $\gamma_2^*:\Gamma(C_{f_1,q}, \Omega^1_{C_{f_1,q}})\to 
   \Gamma(C_{f_1,q}, \Omega^1_{C_{f_1,q}})$.  It follows
   from (\ref{eq:24}) and
   Subsection~\ref{subsec:Serre-duality-differential} that both $1$
   and $-1$ are eigenvalues of $d\gamma_2: \Lie_k(J_{f_1, q}^\nw)\to
   \Lie_k(J_{f_1, q}^\nw)$. Let $e_1:=
   (1+\gamma_2)/2\in \End^0(J_{f_1,q}^\nw)$, and
   $e_2:=(1-\gamma_2)/2\in \End^0(J_{f_1,q}^\nw)$. Both $e_1$ and
   $e_2$ are nontrivial idempotents of $\End^0(J_{f_1,q}^\nw,\ii)$
   with $e_1+e_2=1$. 
   It follows that $\End^0(J_{f_1,q}^\nw,\ii)=E\oplus E$. By
   Proposition~\ref{prop:centralizer-is-e-plus-e},
  \begin{itemize}
  \item if $p\geq 5$ and $q\neq 5, 7$, then $\End^0(J_{f_1,q}^\nw)= E\oplus E$;
  \item if $q=5, 9$, then $\End^0(J_{f_1,q}^\nw)= \Mat_2(E)$; 
  \item if $q=7$, then $\End^0(J_{f_1,q}^\nw)=
    \Mat_3(\qq(\sqrt{-7}))\oplus \qq(\zeta_7)$.
  \end{itemize}


  We claim that $\End^0(J_{f_1,q}^\nw)=E\oplus E$ if $q=3^r\geq
  27$. Let $Y_i=2e_i J_{f_1, q}^\nw$. Then $\Lie(Y_1)$ is the subspace
  of $\Lie(J_{f_1, q}^\nw)$ on which $d\gamma_2$ acts as identity.  On
  the other hand, if $q/3< a <2q/3$, then $dx/y^a$ is
  $\gamma_2^*$-invariant if and only if $a$ is odd. It follows that
  $g_1$ (using notation in
  Subsection~\ref{subsec:case-centralizer-is-E-plus-E}) is given by
  (\ref{eq:38}). By Lemma~\ref{lem:not-isogenous-q-power-of-3},
  $g_2\neq g_1\circ \theta_s$ for any $s\in \umod{q}$, so
  $\End^0(J_{f_1,q}^\nw)=E\oplus E$ as claimed.
 \end{sect}



\begin{sect}\label{subsec:cube-x-plus-one}
  Consider the curve $C_{f_0,N}$ with $f_0(x)=x^3+1$, $3\nmid N$, and
  $N\neq 4, 10$. 
  Let $\gamma_3: C_{f_0,N}\to C_{f_0,N}$ be the automorphism defined by
  $(x,y)\mapsto (\omega x, y)$ for a primitive $3$rd root of unity
  $\omega\in k$.  As $C_{f_0,N}/\dangle{\gamma_3}\cong \pp^1$,
  it follows from (\ref{eq:17}) that $\gamma_3 \in \End
  (J(C_{f_0,N}))$ satisfies the polynomial equation $T^2+T+1=0$. Since
  $\gamma_3$ commutes with $\delta_N$, $J_{f_0,N}^\nw$ is
  $\gamma_3$-invariant, and we have an embedding
  \[ \qq(\zeta_{3N})\cong
  \qq(\zeta_N)[T]/(T^2+T+1)\hookrightarrow \End^0(J_{f_0,N}^\nw),
  \qquad T\mapsto \gamma_3\mid_{J_{f_0,N}^\nw}.\] Now it follows from
  Corollary~\ref{cor:nw-factor-of-jacobian-quad-ext} that
  $J_{f_0,N}^\nw$ is absolutely simple with
  $\End^0(J_{f_0,N}^\nw)=\qq(\zeta_{3N})$.

  This result in itself is not new. A theorem of Kubota-Hazama
  \cite{MR1458755} states that the Jacobian variety of the curve $y^p
  = x^\ell+1$ is absolutely simple if $p$ and $\ell$ are distinct
  primes.  In particular, it follows that if $N=p\neq 3$ is a prime,
  then $J_{f_0, p}^\nw=J(C_{f_0, p})$ is simple. More generally, one
  could realize $C_{f_0, N}$ as a quotient of the Fermat curve
  $X_{3N}: x^{3N}+y^{3N}=1$. Therefore, $J_{f_0, N}^\nw$ is isogenous
  to a factor of $J(X_{3N})$. One checks that it corresponds to the
  triple $(N, 3, 2N-3)\in (\zmod{3N})^3$. There is only one entry
  (namely, $2N-3$) that's coprime to $3N$.  So it is of Type I in
  Aoki's classification, and hence simple if $N>60$ (cf. \cite[Theorem
  0.2]{MR1129293}).
 \end{sect}

 \begin{sect}
   Let $f_0$ and $\gamma_3$ be as in
   Subsection~\ref{subsec:cube-x-plus-one}, and $N=q=3^r>3$. We may
   assume that $\omega=\zeta_q^{3^{r-1}}$.  Let $\delta_3:=
   \delta_q^{3^{r-1}}\in \Aut(C_{f_0,q})$.  By
   (\ref{eq:basis-of-diff-first-kind}), both $dx/y^{q-1}$ and
   $xdx/y^{q-1}$ are differentials of the first kind on $C_{f,q}$, and
   they are eigenvectors corresponding to eigenvalue $\omega$ for
   $\delta_3^*: \Gamma(C_{f_0, q}, \Omega_{C_{f_0,q}}^1)\to
   \Gamma(C_{f_0, q}, \Omega_{C_{f_0,q}}^1)$.¡¡ On the other hand,
   $\gamma_3^* (dx/y^{q-1})= \omega (dx/y^{q-1})$, and $\gamma_3^*
   (xdx/y^{q-1})= \omega^2 (xdx/y^{q-1})$.  In other words, they
   corresponds to distinct eigenvalues for $\gamma_3^*$.  Let
   \[ e_1= \frac{1}{3}(1+\delta_3^2\gamma_3+\delta_3\gamma_3^2),
   \qquad e_2= \frac{1}{3}(1+ \delta_3\gamma_3+
   \delta_3^2\gamma_3^2) \] be elements of $\End^0(J_{f_0,
     q}^\nw)$. Using the fact that both $\delta_3$ and $\gamma_3$
   satisfy the equation $T^2+T+1=0$, one sees that $e_1$ and $e_2$
   are orthogonal idempotents with $e_1+e_2=1$, and it follows from
   (\ref{eq:24}) and
   Subsection~\ref{subsec:Serre-duality-differential} that neither
   $e_1$ nor $e_2$ is zero. Therefore, $\End^0(J_{f_0, q}^\nw,
   \ii)=E\oplus E$.  Let $Y_1:= 3e_1 J_{f_0, q}^\nw$, and $Y_2:= 3e_2
   J_{f_0, q}^\nw$. Then $\Lie(Y_1)$ is the subspace of
   $\Lie(J_{f_0,q}^\nw)$ on which $d(\delta_3^2\gamma_3)$ acts as
   identity.  If $q/3<a<2q/3$ and $3\nmid a$, then $dx/y^a$ is
   invariant under $(\delta_3^2\gamma_3)^*$ if and only if $a\equiv
   2\pmod{3}$.  We see that $h=g_1+g_2$ with $g_1$ given by
   (\ref{eq:29}). Since $g_2=g_1\circ \theta_s$ with $s=3^{r-1}-1$,
   $Y_2$ is isogenous to $Y_1$, and $\End^0(J_{f_0,
     q}^\nw)=\Mat_2(E)$.

   Alternatively, one can see that $\End^0(J_{f_0, q}^\nw)=\Mat_2(E)$
   in the following way. Without lose of generality, assume that
   $k=\cc$. Let $X_q$ be the Fermat curve $x^q+y^q=1$. There exists an
   cover $X_q\to C_{f_0,q}$ given by $(x,y)\mapsto (-x^{3^{r-1}},
   y)$. Therefore, $J_{f_0, q}^\nw$ is isogenous to a factor of
   $J(X_q)$. Using notations of \cite{MR511556}, one sees that
   $J_{f_0, q}^\nw$ is isogenous to $\cc/L_{r,s,t}\times
   \cc/L_{r',s',t'}$ with $(r, s, t)=(2\cdot 3^{r-1}, 1, 3^r- 2\cdot
   3^{r-1}-1)\in (\zmod{q})^3$ and $(r', s', t')=( 3^{r-1}, 1, 3^r-
   3^{r-1}-1)\in (\zmod{q})^3$.  Since $(ur,us,ut)$ coincides with
   $(r', s',t')$ up to permutation with $u=3^{r-1}+1\in \umod{q}$,
   $\cc/L_{r,s,t} \sim \cc/L_{r',s',t'}$.
 \end{sect}

\section{Arithmetic Results}
\label{sec:arithmetic-results}
In this section, we prove the arithmetic results mentioned
Section~\ref{sec:compl-abel-vari}.  For two real numbers $x\leq y$,
let $[x, y]_\zz$ be the set of integers \[ [x, y]_\zz:=\{z\in \zz\mid
x\leq z \leq y, \gcd(z,N)=1\}.\] Through out this section, $n\geq 3$
is an integer that's not a multiple of $N$, and $h$ denotes the
function
\[h: \umod{N}\to \nn, \qquad a\mapsto \fl{\frac{na}{N}},\] where $a$
is taken to be in $[1,N-1]_\zz$ for the floor function.  We have
\begin{equation}
  \label{eq:4}
 h(a)+h(-a)= n-1.
\end{equation}
We are particularly interested in the case $n=3$, where
\[h(a)=
\begin{cases}
  0 \qquad &\text{ if } \quad a\in [1, N/3]_\zz;\\
  1 \qquad &\text{ if } \quad a\in [N/3,2N/3]_\zz;\\
  2 \qquad &\text{ if } \quad a\in [2N/3,N-1]_\zz.
\end{cases}
\]

\begin{prop} \label{prop:exist-value-one} Suppose that $n=3$, and
  $N\nmid n$. There exists $a\in \umod{N}$ such that $h(a)=1$ if and
  only if $N\not\in \{4, 6, 10\}$. 
\end{prop}
\begin{proof}
  For one direction, one easily checks by direct calculation that
  $[N/3, 2N/3]_\zz=\emptyset$ if $N\in \{4, 6, 10\}$. For the other
  direction, we need to show that there exists an $a\in [N/3,
  2N/3]_\zz$ if $N\not\in \{4, 6, 10\}$.  The proof will be separated
  into a few cases based on the factorization type of $N$.

Suppose $N=p$ is prime. If $N=2$, we take $a=1$; if $N=5$, we take
$a=2$; otherwise, $N\geq 7$, so there exists an integer in $[N/3, 2N/3]_\zz$.  

If $N=9$, we take $a=4$. Suppose that $N=3N_0$ with $N_0\geq 4$.  If
$N_0 \not\equiv 2 \pmod{3}$, then let $a=N_0+1$; if $N_0 \equiv 2
\pmod{3}$, let $a=N_0+3$.

Now suppose that $N$ is not a prime and $\gcd(N, 3)=1$. Let $p_0$ be the
smallest prime that divides $N$, $m:=N/p_0$ and $a_0:=\fl{2p_0/3}$.
We will choose appropriate $b>0$ such that $a:= a_0m+b$ lies in $[N/3, 2N/3]_\zz$.  We note that
\begin{equation}
  \label{eq:7}
\frac{N}{3}<\frac{2N-2m+3b}{3} \leq m\fl{\frac{2p_0}{3}}+b \leq
\frac{2N-m+3b}{3}. 
\end{equation}


 By our choice of $p_0$, every prime factor of $m$ is greater or equal
 to $p_0$. In particular, if $m$ is even, then $p_0=2$, and $4\mid
 N$. We also note that $m>3$. Indeed, if $m=2$, then $N=4$,
 contradiction to our assumption; $3\nmid m$ because that $N$ is
 assumed to be coprime to $3$.

 If $p_0^2\mid N$, choose $b=1$, then $m>3b$, so $a<2N/3$ by
 (\ref{eq:7}). Moreover $a\equiv 1 \pmod{p}$ for all $p\mid
 N$. Therefore, $a \in [N/3, 2N/3]_\zz$.

If $N=2m$ with $m$ odd, then $a_0=\fl{(2\cdot 2)/3}=1$, and $m\geq 7$
since $N\neq 10$. We choose $b=2$ so $a=m+2$ is odd, and for any prime
$p\mid m$, $a\equiv 2 \pmod{p}$.  Since $m> 3b=6$, $a<2N/3$ by
(\ref{eq:7}). If $N=p_0m$ with $\gcd(p_0, m)=1$ and $p_0\geq 5$, then
$m\geq 7$, and we choose $b$ in the two element set $\{1, 2\}$ such
that $p_0\nmid a$. For any $p\mid m$, $a \equiv b \pmod{p}$. Hence
$\gcd(a, N)=1$. Since $m\geq 7> 3b$, $a<2N/3$ by (\ref{eq:7}).
\end{proof}



A complex valued function $g$ on $\umod{N}$ is said to \textit{odd} if
$g(-a)=-g(a)$, and \textit{even} if $g(-a)=g(a)$.  So
$h_\odd:=h-(n-1)/2$ is an odd function by (\ref{eq:4}).  For $s\in
\umod{N}$, we write $\theta_s: \umod{N}\to \umod{N}$ for the
multiplication by $s$ map: $a\mapsto sa$.  Recall that a function
$g:\umod{N}\to \cc$ is said to be \textit{primitive} if $g\circ
\theta_s = g \Leftrightarrow s=1$.  We are going to show that $h$ is
primitive if $\gcd(n,N)=1$. Clearly, it is enough to show that
$h_\odd$ is primitive.


\begin{sect}\label{sec:fourier_expansion}
For each $a\in \zmod{N}$, we write
$\dbracket{a}$ for the unique representative of $a$ in $[1, N-1]_\zz$. Then 
\[h_\odd=\fl{\frac{na}{N}}-\frac{n-1}{2}=
\frac{n\dbracket{a}-\dbracket{na}}{N}-\frac{n-1}{2}.\] Let $\VZ_\odd$ denote the
space of all complex valued odd functions on $\umod{N}$. A character
$\chi: \umod{N}\to \cc^\times$ is odd if and only if $\chi(-1)=-1$.
The set of odd characters of $\umod{N}$ forms a basis of
$\VZ_\odd$. Therefore $h_\odd$ can be uniquely written as a linear
combination $\sum c_\chi \chi$ of the odd characters. Suppose that
$\gcd(n,N)=1$.  By the orthogonality of the characters,
  \[
  \begin{split}
    c_\chi&= \frac{1}{\varphi(N)}\sum_{a\in \umod{N}}
    h_\odd(a)\overline{\chi(a)} \\
    &=\frac{1}{\varphi(N)}\sum_{a\in \umod{N}}
    \frac{(n\dbracket{  a}-\dbracket{  na})\overline{\chi(a)}}{N} \quad \text{
      since } \sum_{a\in \umod{N}} \chi(a)=0\\
    &=\frac{n}{\varphi(N)}\sum_{a\in \umod{N}}
    \frac{\dbracket{  a}\overline{\chi(a)}}{N}-\frac{\chi(n)}{\varphi(N)}\sum_{a\in \umod{N}}
    \frac{\dbracket{  na}\overline{\chi(na)}}{N}\\
    &=\frac{n-\chi(n)}{\varphi(N)}\sum_{a\in \umod{N}}
    \frac{\dbracket{  a}\overline{\chi(a)}}{N}
  \end{split}
\] 
Recall that the generalized Bernoulli number 
$B_{1, \chi}$ is defined (cf. \cite[Chapter 4]{MR1421575}) to be 
\[ B_{1, \chi}=\frac{1}{N} \sum_{a\in \umod{N}} \dbracket{a}\chi(a).\]
Therefore, 
\begin{equation}
  \label{eq:fourier}
  c_\chi=(n-\chi(n))B_{1,\bar{\chi}}/\varphi(N).
\end{equation}
Since $n \geq 2$, $n-\chi(n)\neq 0$. It follows that $c_\chi=0$ if and
only if $B_{1, \bar{\chi}}=0$.
\end{sect}

\begin{prop} \label{prop:arithmetic-results-rigidity-of-h}
If $\gcd(n,N)=1$, then $h\circ \theta_s=h$ if and only if $s=1$. 
\end{prop}

\begin{proof} We may and will assume that $N> 2$ since $\umod{2}$ is
  trivial.  Clearly, the proposition is true if and only if it is true
  for $h_\odd$. Moreover, if $h_\odd \circ
  \theta_s=h_\odd$, then $h_\odd\circ \theta_{s^i}= h_\odd$ for all
  $i\in\zz$. Note that 
  \[ h_\odd(1)= \fl{\frac{n}{N}}-\frac{n-1}{2}\leq
  \frac{n-1}{N}-\frac{n-1}{2}<0.\] So $h_\odd(-1)=
  -h_\odd(1)>0$. Therefore $-1$ is not in $\dangle{s}$, the cyclic
  subgroup of $\umod{N}$ generated by $s$.

  Recall that $h_\odd= \sum_{\chi(-1)=-1} c_\chi \chi$.  So 
\[ h_\odd\circ \theta_s = \sum_{\chi(-1)=-1} c_\chi \chi(s) \chi.\]
By the linear independence of characters, we see that $h_\odd\circ
\theta_s =h_\odd$ if and only if 
\begin{equation}
  \label{eq:comparing_fourier_coefficient}
 c_\chi = c_\chi \chi(s)
\end{equation}
for all odd characters $\chi$.
Following \cite[Proposition, p.~1190]{MR511556}, we let $S(N)$ be the
set of odd characters of $\umod{N}$, and 
\[ S_0(N)= \{\chi \in S(N) \mid B_{1, \chi}=0\} \subset S(N). \] Since $B_{1,
  \bar{\chi}}= \bar{B}_{1, \chi}$, $S_0(N)$ coincides with the
set of odd characters $\chi$ with $c_\chi=0$ by
Subsection~\ref{sec:fourier_expansion}.  Further, let $T_1(s,N)$ be
 set 
\[\{ \chi\in S(N)\mid  \chi(s)=1\}  \subset S(N).\]

Assume that $h_\odd\circ \theta_s=h_\odd$.  It follows from
(\ref{eq:comparing_fourier_coefficient}) that 
\begin{equation}
  \label{eq:union_set_equal}
  S_0(N)\cup T_1(s,N)= S(N).
\end{equation}
Our goal is to show that when $s\neq 1$, the size of both $S_0(N)$ and
$T_1(s,N)$ are small comparing to that of $S(N)$, and thus obtain
a contradiction. 

Let $\ord(s, N)$ denote the order of $s$ in $\umod{N}$. Since $-1$ is
not in the cyclic group generated by $s$, 
\[\abs{T_1(s,N)}= \frac{\varphi(N)}{2\ord(s,N)}=\frac{\abs{S(N)}}{\ord(s,N)}.\]
In particular, if $s\neq 1$, then 
\[ \abs{T_1(s,N)} \leq \abs{S(N)}/2.\]
On the other hand, if $\gcd(N, 6)=1$, then 
\[ \abs{S_0(N)}< {\abs{S(N)}}/6\]
by \cite[Proposition, p.~1190]{MR511556}.  This clearly contradicts
(\ref{eq:union_set_equal}). For a general $N$, the proposition follows
if we can prove that $\abs{S_0(N)}<\abs{S(N)}/2$. 
\end{proof}

\begin{lem}
  For any integer $N\geq 3$, $\abs{S_0(N)}<\abs{S(N)}/2$.
\end{lem}
\begin{proof}
  The proof is modeled after that of \cite[Proposition,
  p.~1190]{MR511556}. First, let $\chi$ be an odd primitive character
  with conductor $N$, then (cf. \cite[Chapter 4]{MR1421575})
  \[ L(1,\chi)= \frac{\pi i\tau(\chi)}{N} B_{1,\bar{\chi}},\] where
  $\tau(\chi)=\sum_{a=0}^{N-1}\chi(a)e^{2\pi i a/N}$ is the Gauss
  sum. It is known classically that $\abs{\tau(\chi)}=\sqrt{N}$, and
  $L(1,\chi)\neq 0$. Therefore. $B_{1,\chi}\neq 0$.

  More generally, let $\chi$ be a character modulo $N$ with conductor
  $N_0$, and $\chi_0$ be the character modulo $N_0$ that induces
  $\chi$. Then 
  \[ B_{1, \chi}=B_{1, \chi_0} \prod (1-\chi_0(p)),\] where the
  product is over all prime factors $p$ of $N$ that do not divide
  $N_0$. So $B_{1, \chi}=0$ if and only if there exists a prime factor
  $p$ of $N$ such that $p\nmid N_0$ and $\chi_0(p)=1$. In particular,
  if $N$ is a prime power, then $B_{1,\chi}\neq 0$ for all odd
  characters $\chi$.  Suppose that $N=\prod_{i=1}^m p_i^{e_i}$. For a
  fixed prime divisor $p_i$, the number $u(p_i,N)$ of all odd characters
  $\chi_0$ modulo $N_i:=N/p_i^{e_i}$ with $\chi_0(p_i)=1$ is
\[ u(p_i,N)=
\begin{cases}
  0  \qquad & \text{if } p_i^c \equiv -1 \pmod{N_i} \text{ for some } c
  \in \nn.  \\
  \varphi(N_i)/(2\ord(p_i, N_i)), \qquad & \text{otherwise. }
\end{cases}
\] 

Let $v(p_i,N) = 2u(p_i,N)/\varphi(N)$.  Then 
\[s(N):=\frac{\abs{S_0(N)}}{\abs{S(N)}}\leq \sum_{i=1}^m v(p_i,N) \leq
\sum_{i=1}^m \frac{1}{\varphi(p_i^{e_i})\ord(p_i, N_i)}. \]
We write $w(N)$ for the last sum. Note that it makes sense to talk
about $u, v,w$ only if $N$ has at least two distinct prime factors. 

Here is a  list of some simple facts about $w(N)$. \\
\begin{itemize}
\item  By \cite[Proposition, p.~1190]{MR511556}, $w(N)<1/6$ if $\gcd(N,
6)=1$.\\
\item  $w(M)\geq w(N)$ if $M\mid N$ and $p\mid N \Leftrightarrow p\mid M$.\\
\item  Suppose that $M$ has at least two factors, and $M\mid N$, then  
\[ w(N)\leq w(M)+\sum_{p\mid N, p \nmid M} v(p, N).\]
\end{itemize}

We separate the estimate of $w(N)$ into cases according to the factor
type of $N$.  \\

Case 1. Assume that $N$ has at least two distinct prime
divisors that's greater or equal to $5$. Since $\ord(p, N_i)\leq
\fl{\log_p N_i}+1$, we have
\[  w(N)< \frac{1}{\varphi(2)(\fl{\log_2 (5\cdot
    7)}+1)}+\frac{1}{\varphi(3)(\fl{\log_3(5\cdot 7)}+1)}+ \frac{1}{6}=\frac{11}{24}<\frac{1}{2}\]

Case 2. Assume that $N>42$, and $N=2^{e_1}3^{e_2}p^{e_3}$ with $e_1, e_2, e_3\geq 0$ and
$p\geq 5$. 
\begin{align*}
\text{ If } e_1\geq 1,&\qquad   \varphi(2^{e_1})\ord(2, N/2^{e_1})\geq \max_{e_1\geq
    1}\big\{\varphi(2^{e_1})(\fl{\log_2(N/2^{e_1})}+1)\big\}\geq 5.\\
\text{ If } e_2\geq 1, &\qquad \varphi(3^{e_2})\ord(3, N/3^{e_2})\geq \max_{e_2\geq
  1}\big\{\varphi(3^{e_2})(\fl{\log_3(N/3^{e_2})}+1)\big\}\geq 6.
\end{align*}
Moreover, if $e_3\geq 1$, then 
$\varphi(p^{e_3})\ord(p, N/p^{e_3})\geq 8$ since either $\varphi(p^{e_3})\geq 8$,
or $e_3=1$ and $p=5$ or $7$, and $\ord(p, N/p^{e_3})\geq 2$. 
Overall, we see that 
\[ w(N)\leq \frac{1}{5}+\frac{1}{6}+\frac{1}{8}=\frac{59}{120}<\frac{1}{2}.\]
  
Case 3. Assume that $N\leq 42$. 

If $N=6$, then $\umod{6}\cong \zmod{2}$. There is a unique odd
character $\chi$ modulo 6, and $B_{1,\chi}= (1-5)/6\neq 0$.

If $N=2\ell^{e_2}$ for some odd prime $\ell$ with $\ell^{e_2}\geq 5$,
then $v(\ell, N)=0$ since $\ell \equiv -1 \pmod{2}$. So
\[ s(N)\leq v(2, N)\leq  \frac{1}{\varphi(2)(\fl{\log_2 \ell^{e_2}}+1)}
\leq \frac{1}{3}.\]

If $N=4\cdot \ell^{e_2}$ with $e_2\geq 1$ and $\ell=3,7$, then $v(\ell,
N)=0$ since $\ell\equiv -1 \pmod{4}$. So
\[ s(N)\leq v(2, N)\leq
\frac{1}{(\varphi(4)(\fl{\log_2\ell^{e_2}}+1))}\leq \frac{1}{4}. \]

If $N=2^{e_1}\cdot 5$, then $v(2,N)=0$ since $4\equiv -1 \pmod{5}$. So 
\[ s(N)\leq v(5, N) \leq \frac{1}{\varphi(5)} =\frac{1}{4}.\]

If $N=24$, then $v(2, 24)=1/(2\varphi(8))=1/8$, and $v(3,
24)=1/(2\varphi(3))=1/4$. So $s(24)\leq w(24)=1/8+1/4<1/2$. 

If $N=3p$ with $p\geq 5$, then $\ord(3,p)\geq 3$ since $p\nmid
(3^2-1)$. So 
\[ w(N)\leq \frac{1}{3\varphi(3)}+\frac{1}{\varphi(p)}\leq
\frac{1}{4}+\frac{1}{6}<\frac{1}{2}.\]

If $N=30$, then $v(2, 30)=1/4$, and $v(3,30)=0$ since $3^2\equiv -1
\pmod{10}$, and $v(5,30)=0$ since $5\equiv -1 \pmod{6}$. Therefore, 
$s(30)\leq 1/4$. 

If $N=42$, then $\ord(2,21)=6$, $v(3, 42)=0$ since $3^3\equiv -1
\pmod{14}$, and $v(7, 42)=1/6$. So $s(42)\leq 1/6+1/6 = 1/3$. 

This complete the enumeration of all the positive numbers $N\leq 42$
that are not prime powers.
\end{proof}

For the rest of the section, we focus on the case where $N=q=p^r$ is a
prime power and $n=3$. Then $h(a)=\fl{3a/q}$.  Recall that
\begin{gather*}
    \TZ_q:=\{g: \umod{q}\to \{0, 1\}\mid g(a)+g(-a)=1,  g([1, q/3]_\zz)=0 \};\\
 \SZ_q:=\{s\in \umod{q} \mid s\neq 1, \text{and } \exists g\in
  \TZ_q \text{ such that } g\circ \theta_s\in \TZ_q\}.     
\end{gather*}
Let $g_i: \umod{q}\to \{0, 1\}$ be functions satisfying
$g_i(a)+g_i(-a)=1$ for $i=1,2$. If $h=g_1+g_2$, then $g_i\in \TZ_q$
for all $i$.

\begin{sect}\label{subsec:elementary-elimination-for-s}
We note that if $p$ is odd, then $2\in \SZ_q$. Indeed, let $g\in
\TZ_q$ be the function given by 
\begin{equation}
  \label{eq:38}
 g(a)=
 \begin{cases}
   0    &\qquad     \text{ if } \quad a \in [1, q/3]_\zz;\\
   0    &\qquad     \text{ if } \quad a \in [q/3, 2q/3]_\zz \text{ and } a
   \text{ is even};\\ 
   1   &\qquad     \text{ if } \quad a \in [q/3, 2q/3]_\zz \text{ and } a
   \text{ is odd};\\ 
   1    &\qquad     \text{ if } \quad a \in [2q/3, q]_\zz,\\
 \end{cases}
\end{equation}
then $g\circ \theta_2\in \TZ_q$. On the other hand, if $g'\in \TZ_q$
is any function such that $g'\circ \theta_2\in \TZ_q$, then $g'(a)=0$
for all even $a \in [q/3, 2q/3]_\zz$. Since $q$ is $odd$, it follows
from the assumption $g'(a)+g'(-a)=1$ that $g'(a)=1$ for all odd $a\in
[q/3, 2q/3]_\zz$. So $g'=g$. In other words, the function with both
$g$ and $g\circ \theta_2$ in $\TZ_q$ is uniquely determined. Since
$2^{-1}=(q+1)/2\in \umod{q}$, $(q+1)/2\in \SZ_q$ as well.  On the
other hand, $-1\not\in \SZ_q$ since $(g\circ\theta_{-1})(1)=g(q-1)=1$.
\end{sect}



\begin{lem}\label{lem:possible_choices_of_s}
Suppose that $q=p^r$ with $p\geq 5$, then $\SZ_q=\{2, (q+1)/2\}$. 
\end{lem}
\begin{proof}
  Note that the lemma is trivial for $q=5$ since $(q+1)/2=3$ in this
  case, so the only other nontrivial element in $(\zz/5\zz)^\times $
  that's not in our list is $s=4 \equiv -1 \pmod{5}$, which is not in
  $\SZ_q$ as remarked in
  subsection~\ref{subsec:elementary-elimination-for-s}.  Henceforth we
  assume that $q\geq 7$. In particular, $2\in [0, q/3]_\zz$. Suppose
  that $g\in \TZ_q$ is a function such that $g\circ \theta_s\in
  \TZ_q$. We narrow down on the possible $s$ in steps.
  \begin{step}
$s\in [1, 2q/3]_\zz$.  Otherwise, $(g\circ \theta_s)(1)=g(s)=1$. 
  \end{step}
  \begin{step}
    We show that the lemma is true for $7\leq q\leq 23$. This will
    also give us a glimpse of the idea of the proof for larger
    $q$. Note that $q$ is a prime in this case.  First suppose that
    $3\leq s <q/3$, then $2<\fl{q/s}<q/3$, and
    \[ q> s\fl{\frac{q}{s}} \geq q-s+1> \frac{2q}{3}+1.\] So it
    follows that $g(s\cdot \fl{q/s})=1$. Contradiction.

    Hence we must have $s\in [q/3, 2q/3]_\zz$. If $q/3< s< q/2$, then
    $ 2q/3< 2s <q$, and $(g\circ\theta_s)(2)=g(2s)=1$. Contradiction
    again!

Therefore $s\in [q/2, 2q/3]_\zz$. Note that the lemma is already
proved for $q=7$ since the only element in the set $[7/2, 14/3]_\zz$
is $4=(7+1)/2$. We further assume that $q\geq 11$. In particular $3\in
[1, q/3]_\zz$ and thus $g(3)=0$. Note that $3q/2 < 3s < 2q$. In order
that $g(3s)=0$ we must have $3s-q< 2q/3$, or equivalently $s <
5q/9$. Recall that $q\leq 23$. So $(q+3)/2 > 5q/9$, and the only
element in $[q/2, 5q/9]_\zz$ is $(q+1)/2$. The lemma is proved for
all $7\leq q \leq 23$. 
  \end{step}
  We assume that $q\geq 25$ for the rest of the proof. In particular,
  both $3$ and $4$ are in $[1, q/3]_\zz$.
  \begin{step}
    In this step we will show that if $s\geq 3$, then $s\in [q/2,
    2q/3]_\zz$. One difference with the previous step is that $q$ is
    not necessarily prime, so we have to avoid using any $1\leq a<
    q/3$ that are divisible by $p$ in our proof. 

    First, we claim that $s\not\in [3, q/6]_\zz$. Suppose otherwise,
    then
\begin{gather*}
  0<\fl{\frac{q}{s}}-1<\fl{\frac{q}{s}}\leq \fl{\frac{q}{3}},\\
  q> s\fl{\frac{q}{s}}> s
  \left(\fl{\frac{q}{s}}-1\right)> q-2s >
  \frac{2q}{3}.
\end{gather*}
If $p\nmid \fl{q/s}$, we set $a=\fl{q/s}$, otherwise, we set
$a=\fl{q/s}-1$. Then $a\in [2, q/3]_\zz$, and $a\cdot s\in [2q/3, q]_\zz$,
so
\[ (g\circ \theta_s)(a)=g(a\cdot s)= 1.\] This
contradicts the assumption of the lemma.

Second, we claim that $s\not\in [q/3,\; q/2]_\zz$. Suppose this
is not true, then $ 2q/3<2s<q$, and
$(g\circ\theta_s)(2)=g(2s)=1$. Contradiction.

Third, if $2q/9<s<q/3$, then $2q/3<3s<q$, and it follows that
\[ (g\circ\theta_s)(3)=g(3s)=1.\] Again, this leads to a
contradiction.

Last, if $q/6<s<2q/9$, then $2q/3<4s<8q/9<q$. Once again the
contradiction arises since $(g\circ\theta_s)(4)=g(4s)=1$. So we must
have
\begin{equation}
  \label{eq:2}
  q/2 < s < 2q/3.
\end{equation}
  \end{step}
  \begin{step}
    Assume that $s\neq 2$. Then we must show that $s=(q+1)/2$. By
     (\ref{eq:2}), $3q/2< 3s < 2q$.  Since $g(3s)=g(3s-q)=0$,
    we must have $3s-q<2q/3$, i.e.,
\begin{equation}
  \label{eq:3}
 q/2< s < 5q/9. 
\end{equation}
Note that $5/9< 2/3$, so the upper bound for $s$ has been lowered.
Now the idea is to repeat this process by taking the products of $s$
with odd numbers $a=5, 7, 9$, etc., and inductively lower the upper
bound until there is no other element left in the interval except
$s=(q+1)/2$.  But once again extra care must be taken since we need to
make sure that the odd numbers that are divisible by $p$ be skipped.

For integers $t\geq 1$, let
\[c_t= \frac{3t-1}{3(2t-1)}=\frac{1}{2}+\frac{1}{6(2t-1)}.\] Note that
$c_1=2/3,\; c_2=5/9$, and $c_t>c_{t+1}$. Suppose that for a given
$t\geq 2$ we have \[ q/2< s < c_t q.\] Then 
\begin{gather*}
   tq < (2t+1)s <
\frac{(2t+1)(3t-1)q}{3(2t-1)}\leq (t+1)q;\\
   (t+1)q < (2t+3)s < \frac{(2t+3)(3t-1)q}{3(2t-1)}< (t+2)q.
\end{gather*}
Now assume additionally that $2t+3<q/3$. If $p\nmid 2t+1$, then
$(g\circ\theta_s)(2t+1)=g((2t+1)s)=0$. Hence $(2t+1)s-tq < 2q/3$. That is,
\[q/2<s < \frac{(3t+2)q}{3(2t+1)}= c_{t+1}q.\] Similarly, if $p\nmid
2t+3$, we have $(2t+3)s-(t+1)q<2q/3$, and hence
\[q/2<s < \frac{(3t+5)q}{3(2t+3)}= c_{t+2}q.\]

Clearly, either $2t+1$ or $2t+3$ is not divisible by $p$. Recall that
$c_{t+2}< c_{t+1}$. We see that as long as $2t+3<q/3$,
\[ s\in [q/2, c_tq]_\zz \quad \Longrightarrow \quad s\in [q/2,
c_{t+1}q]_\zz.\] Let $m=[(q-9)/6]>(q-14)/6$. The base case for
$t=2$ is already verified in (\ref{eq:3}). It follows by induction
that $s\subseteq [q/2, c_{m+1}q]_\zz$. Note that
\[ c_{m+1}q-q/2= \frac{q}{6(2m+1)}<
\frac{q}{(2(q-14)+6)}=\frac{q}{2q-22}<1\] because $q\geq 25$ by
assumption, so the only integer in $[q/2, c_{m+1}q]_\zz$ is
$(q+1)/2$. Therefore, $s=(q+1)/2$. 
  \end{step}
Lemma~\ref{lem:possible_choices_of_s} follows by combining all the
above steps. 
\end{proof}
\begin{cor}\label{cor:primitive_cm_type}
  Suppose that $p\geq 5$, and $q$ is not $7$. For any $g\in \TZ_q$,
  $g\circ \theta_s=g$ if and only if $s=1$. 
\end{cor}
\begin{proof}
  Indeed, if $g=g\circ \theta_s$ and $s\neq 1$, then all nontrivial
  elements of the cyclic group $\dangle{s}$ lie in the two element set
  $\SZ_q=\{2, (q+1)/2\}$.  It follows that $4=2^2\in \SZ_q$. This is
  possible only if $q=7$ whence $(q+1)/2=4$ and $2\in \umod{7}$ has
  order 3.
\end{proof}


\begin{rem}\label{rem:case-q-is-7}
  If $q=7$, then up to relabeling, $h$ can be uniquely written as
  $g_1+g_2$, where $h, g_1, g_2$ are given by :
  \begin{center}
  \begin{tabular}{|c||c|c|c|c|c|c|c|}
\hline
    & 1 & 2 & 3 & 4 & 5 & 6 \\
\hline
  $h$ & 0 & 0 & 1 & 1& 2 & 2\\
  $g_1$& 0 & 0 & 1 & 0& 1 & 1 \\
  $g_2$& 0 & 0 & 0 & 1 & 1 & 1\\
\hline
  \end{tabular}
  \end{center}
It is clear that $g_1\circ \theta_2= g_1$. 
\end{rem}

\begin{lem}\label{lem:splitting_for_2_twist}
  Suppose that $p$ is an odd prime not equal to $3$, and $q\neq
  5$. There does not exist a function $g:\umod{q}\to \{0, 1\}$ 
 satisfies both of the following conditions:\\
 (i)\phantom{i} $g(a)+g(-a)=1$, $\; \forall\; a\in (\zz/q\zz)^\times$; \\
 (ii) $h= g+g\circ \theta_s$ for some $s\in \umod{q}$.
\end{lem}
\begin{proof}
  We prove by contradiction. Suppose such a function $g$ exits. Since
  $h(a)=0$ for all $a\in [1, q/3]_\zz$, condition \textit{(ii)}
  implies that both $g$ and $g\circ \theta_s$ lies in
  $\TZ_q$. Clearly, $s\neq 1$ since there exists $a\in \umod{q}$ such
  that $h(a)=1$.  Without lose of generality, we may assume that $s=2$
  so $g$ is given by (\ref{eq:38}).  Let us set $a_0= (q-1)/2$ if
  $q\equiv 3 \pmod{4}$; and $a_0= (q-3)/2$ if $q\equiv 1 \pmod{4}$,
  then $a_0$ is odd, $q/3<a_0< q/2$, and $(a_0, p)=1$. But we have
  $g(a_0)=1$, and $(g\circ\theta_2)(a_0)=g(2a_0)=1$.  Therefore,
  $g(a_0)+(g\circ\theta_2)(a_0)=2$. On the other hand,
  $f(a_0)=[3a_0/q]=1$. Contradiction!
\end{proof}




\begin{rem}\label{rem:case-q-is-5}
 Lemma~\ref{lem:splitting_for_2_twist} fails for $q=5$ since
 $h=g_1+g_2$, where $h$, $g_1$ and $g_2$ are given by the following table:
  \begin{center}
  \begin{tabular}{|c||c|c|c|c|c|}
\hline
    & 1 & 2 & 3 & 4 \\
\hline
  $h$ & 0 & 1 & 1 & 2\\
  $g_1$& 0 & 1 & 0 & 1 \\
  $g_2$& 0 & 0 & 1 & 1\\
  $g_2\circ \theta_2$& 0 & 1 & 0 & 1\\
\hline
  \end{tabular}
  \end{center}
  One sees that $g_2\circ \theta_2=g_1$.
\end{rem}

\begin{rem}
\label{subsec:uniqueness-of-g-in-case-q-is-power-of-2}
  Let $q=2^r$ with $r\geq 3$, and $s=2^{r-1}-1$. Then for any positive
  odd number $2t+1<2^{r-1}$, we have 
\[ s(2t+1) = (2^{r-1}-1)(2t+1)= 2^rt+2^{r-1}-(2t+1)\equiv
2^{r-1}-(2t+1) \pmod{2^r}.\]
In particular, $\theta_s([1, 2^{r-1}]_\zz)=[1,
2^{r-1}]_\zz$. Since $\theta_s$ is bijective, $\theta_s([2^{r-1}, 2^r]_\zz)=[2^{r-1}, 2^r]_\zz$. Let $g: (\zz/q\zz)^\times\to \{ 0,1\}$ be the
function defined by
\begin{equation}
  \label{eq:30}
 g(a)=\left\{
\begin{aligned}
  0 \qquad &\text{ if } a\in [1, 2^{r-1}]_\zz, \\
  1 \qquad &\text{ if } a\in [2^{r-1}, 2^r]_\zz.
\end{aligned}\right.
\end{equation}
Then $g\in \TZ_q$ and $g=g\circ \theta_s$ with $s=2^{r-1}-1$. In other
words, Corollary~\ref{cor:primitive_cm_type} fails for all powers of
$2$ that's greater or equal to $8$.

On the other hand, let $g'\in \TZ_q$ be any function such that
$g'\circ \theta_{(q/2-1)}\in \TZ_q$ as well. For any $a\in [q/6,
q/2]_\zz$, we have $g'(a)= (g'\circ \theta_s)(q/2-a)=0$ since
$q/2-a\in [0, q/3]_\zz$. Combining with the fact that $g'$ also
vanishes on $[1, q/3]_\zz$, we see that $g'$ coincides with $g$.  In
conclusion, if $q=2^r\geq 8$ and $s=q/2-1$, there is a unique $g\in
\TZ_q$ such that $g\circ \theta_s\in \TZ_q$.
\end{rem}


\begin{sect}\label{subsec:special-form-of-g-when-s-has-order-6}
Suppose that $q=3^r\geq 9$. We claim that $\SZ_q\supseteq
\{2, (q+1)/2, q/3-1, 2q/3-1\}$. Note that 
\[ (q/3-1)(2q/3-1)= 2q^2/9 - q +1 \equiv 1 \pmod{q}, \] so
$(q/3-1)^{-1}= 2q/3-1\in \umod{q}$. It is enough to show that
$\SZ_q\ni (q/3-1)$. Let $s=q/3-1$, and $a=3t+a_0$ with $a_0=1$ or
$2$. Then
\[ sa= qa/3- a =q(3t+a_0)/3 -a \equiv a_0q/3 -a \pmod{q}.\] Suppose
$a\in [1, q/3]_\zz$. If $a_0=1$, then $\dbracket{sa}= q/3-a\in [1,
q/3]_\zz$; and if $a_0=2$, then $\dbracket{sa} = 2q/3-a\in [q/3,
2q/3]_\zz$. Moreover, if $a\equiv 1\pmod{3}$ and $a\in [q/3,
2q/3]_\zz$, then one easily shows that $\exists a'\in [1, q/3]_\zz$
such that $\dbracket{sa'}=a$. Therefore, if $g\in \TZ_q$ is a function
such that $g\circ \theta_s\in \TZ_q$, then $g$ must be of the form
\begin{equation}
  \label{eq:29}
 g(a)=\left\{
  \begin{aligned}
   0    &\qquad     \text{ if } \quad a \in [1, q/3]_\zz,\\
   0    &\qquad     \text{ if } \quad a\in [q/3, 2q/3]_\zz \text { and
     } a \equiv 1 \pmod{3},\\
   1   &\qquad     \text{ if } \quad a \in [q/3, 2q/3]_\zz \text { and
     } a  \equiv 2 \pmod{3},\\
   1    &\qquad     \text{ if }\quad  a \in [2q/3, q]_\zz.\\
  \end{aligned}\right.
\end{equation}
Last we note that the order of $s= q/3-1\in \umod{q}$ is $6$. Indeed,
since $q\geq 9$, 
\[ s^3= (q/3-1)^3= q^3/27 - 3 (q/3)^2 + 3 (q/3) -1 \equiv -1 \pmod{q}.\]
\end{sect}

\begin{sect}
  Suppose that $s=q/3-1$, and $g$ be as in (\ref{eq:29}) so that
  $g\circ \theta_s\in \TZ_q$. Given $a\in [q/3, 2q/3]_\zz$, if
  $a\equiv 1 \pmod{3}$, then $(g\circ\theta_s)(a)=g(sa)=1$; and if
  $a\equiv 2\pmod{3}$, then $(g\circ \theta_s)(a)=0$.  It follows that
\[ h = g+ g\circ \theta_{(q/3-1)}\]
for this particular $g$. 
\end{sect}

Similar to $[x,y]_\zz$, we define 
\[ (x,y)_\zz:=\{z\in \zz\mid x< z< y, \quad \gcd(z, q)=1\}.\]

\begin{lem}\label{lem:pos_choice_of_s_when_q_is_power_of_3}
Suppose that $q=3^r\geq 9$, then $\SZ_q=\{
2, (q+1)/2, q/3-1, 2q/3-1\}$.  
\end{lem}
\begin{proof}
  The proof is follows the same ideas of that of
  Lemma~\ref{lem:possible_choices_of_s}, except extra care must be
  taken.

  First, if $q=9$, we only need to prove that $s\neq 4$. Indeed,
  $2<9/3=3$, but $(g\circ \theta_4)(2)= g(8)=1$. We assume that $q\geq
  27$ for the rest of the proof. In particular, $q/3\geq 9$.

  Clearly, $s\in [1, 2q/3]_\zz$. Similar arguments as Step 3 of
  Lemma~\ref{lem:possible_choices_of_s} shows that $s\not\in (q/3,
  q/2)_\zz$ and $s\not\in (3, q/6)_\zz$. Suppose that $s\in (q/6,
  q/4)_\zz$, then $2q/3\leq 4s \leq q$, and therefore $(g\circ \theta_s)
  (4)= g(4s)=1$. Contradiction. Suppose that $s\in (q/4, 2q/7)_\zz$,
  then $5q/3 < 7q/4 < 7s < 2q$, and hence $(g\circ \theta_s) (7)=
  g(7s)=1$. Contradiction. We have shown that if $s\in (3, q/2)_\zz$,
  then $s\in (2q/7, q/3)_\zz$.  In particular, if $q=27$, then the
  only integer in $(2q/7, q/3)_\zz$ is $q/3-1=8$. We may further
  assume that $q\geq 81$.

Let $t\in \nn$. Suppose that $s\in (2tq/(6t+1), q/3)_\zz$. Notice that
this is true for $t=1$. Moreover, 
\[ (6t+7)s > \frac{2t(6t+7)q}{6t+1}>
(2t+1)q+\frac{2q}{3}=\frac{(6t+5)q}{3}. \]

Indeed,
\[ \frac{2t(6t+7)}{6t+1}-\frac{(6t+5)}{3}=
\frac{6t(6t+7)-(6t+1)(6t+5)}{3(6t+1)}= \frac{6t-5}{3(6t+1)}>0.\]
Therefore, if $s \in (2tq/(6t+1), 2(t+1)q/(6t+7))_\zz$, then 
\[ (2t+1)q+\frac{2q}{3}    < (6t+7)s < 2(t+1)q.\]
We get a contradiction as long as $6t+7 <q/3$. Since
\[ \frac{2t}{6t+1}= \frac{1}{3} - \frac{1}{3(6t+1)},\] it is an
increasing function in $t$. Our bound are refined each time we
increase $t$. Take $t_0=[(q-3)/18]= (q-9)/18$. Then we get
\[\frac{q}{3}- \frac{2t_0q}{(6t_0+1)}=
\frac{q}{3(6t_0+1)}=\frac{q}{q-6}\leq \frac{27}{25}.\]
It follows that if $s\in (2q/7, q/3)_\zz$, then $s=q/3-1$.

Now assume that $s\in (3q/5, 2q/3)_\zz$. If $q=27$, then the only
element in $(3q/5, 2q/3)_\zz$ is $s=17= 2q/3-1$. So we may assume that
$q\geq 81$ in this case.  Suppose that $s\in (3q/5, 20q/33)_\zz$, then 
\[ 7q+\frac{4q}{5}=\frac{39q}{5}< 13s < \frac{260q}{33}=
7q+\frac{29q}{33}. \] We have $(g\circ \theta_s)
(13)=1$. Contradiction.  Now suppose that $s \in (20q/33, 7q/11)_\zz$,
then
\[  6q+\frac{2q}{3}<   11s < 7q.\]
Therefore, $g\circ \theta_s (11)=1$ but $g(11)=0$. Contradiction
again. So we must have $s \in (7q/11, 2q/3)_\zz$.

Now let $t\in \nn$. Suppose that
$(4t-1)q/(6t-1)<s < 2q/3$. Clearly, this is true for $t=2$. Now 
\[\frac{4t-1}{6t-1}= \frac{2}{3}-\frac{1}{3(6t-1)}.\]
So $\frac{4t-1}{6t-1}$ as a function in $t$ is increasing. 
Moreover, 
\[ \frac{(4t-1)}{6t-1}\cdot (6t+5) > 4t+2+\frac{2}{3}.\]
Indeed, 
\[
\begin{split}
\frac{(4t-1)}{6t-1}\cdot (6t+5) - (4t+2+\frac{2}{3})&=
\frac{3(4t-1)(6t+5)-(6t-1)(12t+8)}{3(6t-1)}  \\
&= \frac{6t-7}{3(6t-1)}> 0 \qquad \text{ if } t \geq 2.
\end{split}
\]
Therefore, if $6t+5\in (1, q/3)_\zz$, then $s\in ((4t-1)q/(6t-1),
2q/3)_\zz$ implies that $s\in ((4t+3)q/(6t+5),
2q/3)_\zz$. 

Let $t_0= [(q+3)/18]= (q-9)/18$. We then have 
\[ \frac{2q}{3}-\frac{(4t_0-1)q}{6t_0-1}=
\frac{q}{3(6t_0-1)}=\frac{q}{q-12}\leq \frac{27}{23}.\]   
This shows the  only possible $s\in (3q/5, 2q/3)_\zz$ is $s= 2q/3-1$. 

We need to handle the remaining case $q/2<s < 3q/5$. Note that $5q/2<
5s < 3q$, so we must have $5q/2< 5s < 8q/3$. That is $q/2< s< 8q/15$. 

Assume that $q/2<s< (9t-1)q/(3(6t-1))$. Then this holds for $t=1$. 
\[ \frac{9t-1}{3(6t-1)}= \frac{1}{2}+\frac{1}{6(6t-1)}.\]
So $\frac{9t-1}{3(6t-1)}$ is a decreasing function in $t$. 
On the other hand, 
\[ \frac{(9t-1)(6t+5)}{3(6t-1)} > (3t+2)+\frac{2}{3}. \]
Indeed, 
\[
\begin{split}
\frac{(9t-1)(6t+5)}{3(6t-1)}- ((3t+2)+\frac{2}{3})&=
\frac{(9t-1)(6t+5)-(9t+8)(6t-1) }{3(6t-1)}\\
&=\frac{1}{(6t-1)}>0.
\end{split}\]
Therefore, if $6t+5< q/3$, then $s\in (q/2,(9t-1)q/(3(6t-1)))_\zz$
implies that $s\in (q/2,(9t+8)q/(3(6t+5)))_\zz$. 

Now we take the largest $t_0= [(q+3)/18]= (q-9)/18$. Then 
\[\frac{(9t_0-1)q}{3(6t_0-1)}-\frac{q}{2}=\frac{q}{6(6t_0-1)}=
\frac{q}{2q-24}\leq \frac{9}{10} \qquad \text{ since } q\geq 27. \]
This takes care all the cases for $q=3^r$. 
\end{proof}


\begin{cor}\label{cor:primitive-cm-type-p-is-3}
  Let $q = 3^r\geq 9$. For any $g\in \TZ_q$, $g\circ \theta_s=g$ if
  and only if $s=1$. 
\end{cor}
\begin{proof}
  Suppose that $g\circ \theta_s=g$ and $s\neq 1$.  By
  Lemma~\ref{lem:pos_choice_of_s_when_q_is_power_of_3}, $s\in
  \SZ_q=\{2, (q+1)/2, q/3-1, 2q/3-1\}$, and all nontrivial elements of
  the cyclic group $\dangle{s}$ lie in $\SZ_q$.  If $s=2$ or
  $(q+1)/2$, then $4\in \SZ_q$. On the other hand, if $q=9$, then
  $\SZ_q=\{2, 5\}$; if $q>9$, then $q/3-1, (q+1)/2$ and $2q/3-1$ are
  all strictly greater than $4$. So $s\not \in \{2, (q+1)/2\}$. If
  $s=q/3-1$ or $2q/3-1$, then the six element group $\langle s\rangle$
  will fix $g$. However, there are at most 4 elements in
  $\SZ_q$. Contradiction.
\end{proof}

\begin{rem}\label{rem:isogeny-q-is-9}
  When $q=9$, up to labeling, there is a unique way to write
  $h=g_1+g_2$, where $h$, $g_1$ and $g_2$ are given by the following
  table:
  \begin{center}
  \begin{tabular}{|c||c|c|c|c|c|c|}
\hline
    & 1 & 2 & 4 & 5 & 7 & 8 \\
\hline
  $h$ & 0 & 0 & 1 & 1 & 2 & 2 \\
  $g_1$& 0 & 0 & 0 & 1 & 1 & 1 \\
  $g_2$& 0 & 0 & 1 & 0 & 1 & 1\\
  $g_1\circ \theta_2$& 0 & 0 & 1 & 0 & 1 & 1\\
\hline
  \end{tabular}
  \end{center}
  One sees that $g_2=g_1\circ \theta_2$.
\end{rem}


\begin{lem}\label{lem:not-isogenous-q-power-of-3}
  Suppose that $q=3^r\geq 27$, and $g\in \TZ_q$ be the function given
  by (\ref{eq:29}).  Let $h=g_1+g_2$ with $g_i\in \TZ_q$.  Then
  $g_2=g_1\circ \theta_s$ for some $s\in \umod{q}$ if and only if the
  pair $(s, g_1)$ coincides with $(q/3-1,g)$ or $(2q/3-1, g\circ
  \theta_{(q/3-1)})$. In particular, $g\not\in \{ g_1, g_2\}$, then
  $g_2\neq g_1\circ \theta_s$ for any $s\in \umod{q}$.
\end{lem}
\begin{proof}
  By subsection~\ref{subsec:special-form-of-g-when-s-has-order-6}, the
  pairs $(q/3-1,g)$ and $(2q/3-1, g\circ \theta_{(q/3-1)})$ satisfies
  the required conditions. On the other hand, it follows from
  Lemma~\ref{lem:pos_choice_of_s_when_q_is_power_of_3} that it is
  enough to show that there does not exists $g'\in \TZ_q$ such that
  $h=g'+g'\circ \theta_2$. Suppose that $g'$ is such a function.  By
  subsection~\ref{subsec:elementary-elimination-for-s}, it must be of
  the form given by (\ref{eq:38}) in order that $g'\circ \theta_2\in
  \TZ_q$. Let $a_0= q/3+2$. Then $3\nmid a_0$, $a_0$ is odd, and
  $2q/3<2a_0<q$ since $q\geq 27$. Therefore, $g'(a_0)=1$, and
  $(g'\circ \theta_2)(a_0)= g'(2a_0)= 1$.  On the other hand,
  $h(a_0)=1$, so $h(a_0)\neq g'(a_0)+ (g'\circ \theta_2)(a_0)$.
\end{proof}

\begin{rem}\label{rem:case-q-is-8}
If $q=8$, then $h$ can be uniquely written as $g_1+g_2$ (up to
relabeling) with $h, g_1, g_2$ given by 
  \begin{center}
  \begin{tabular}{|c||c|c|c|c|}
\hline
    & 1 & 3 & 5 & 7   \\
\hline
  $h$ & 0 & 1 & 1 & 2  \\
  $g_1$& 0 & 1 & 0 & 1  \\
  $g_2$& 0 & 0 & 1 & 1 \\
\hline
  \end{tabular}
  \end{center}
One easily checks that $g_1=g_1\circ \theta_5$, and $g_2=g_2\circ
\theta_3$. 
\end{rem}

\begin{lem}\label{lem:q-is-a-power-of-2}
Let $q=2^r\geq 16$, then $\SZ_q=\{q/2-1\}$.  
\end{lem}

\begin{proof}
Clearly, $s\in [1, 2q/3]_\zz$. If $q=16$, we need to show that $s\not
\in \{ 3, 5, 9\}$.  $s\neq 3$ or $5$ since $15=q-1\in [2q/3, q]_\zz$,
and both $3, 5\in [1, q/3]_\zz$. On the other hand, if $s=9$, then
$3s=27\equiv 11 \pmod{16}$, but $11\in [2q/3, q]_\zz$. Contradiction.

Assume that $q\geq 32$.  Suppose that $s\in [3, q/6]_\zz$ we treat it
similarly as the previous case.  Suppose that $s\in [2q/9, q/3]_\zz$,
$2q/3< 3s <q$. Contradiction.  Suppose that $s\in [q/6, 2q/9]_\zz$, then
$3q/2< 9s < 2q$. Therefore, $s\in [q/6, 5q/27]_\zz$.  Now $5q/6< 5s <
25q/27$. Contradiction.

Suppose that $s\in [q/3, q/2]_\zz$. Then $5q/3 < 5s <
5q/2$. Therefore, we must have $ 2q<5s< 5q/2$, that is $s\in [2q/5,
q/2]$. Suppose that for some $t\in \nn$ we have $tq/(2t+1)< s < q/2$. Then 
\[ \frac{t}{2t+1}= \frac{1}{2}- \frac{1}{2(2t+1)}.\]
So $t/(2t+1)$ is an increasing function in $t$. 
If $2t+3<q/3$, then 
\[  \frac{(2t+3)t}{2t+1}- (t+\frac{2}{3}) =
\frac{3t(2t+3)-(2t+1)(3t+2)}{3(2t+1)}= \frac{2t-2 }{3(2t+1)}>0.\] 
Therefore, 
\[ (t+ \frac{2}{3})q <  (2t+3)s < (t+3/2)q.\]
Therefore, $   (t+1)q< (2t+3)s < (t+3/2)q$, and hence $s\in
[(t+1)q/(2t+3), q/2]_\zz$. 

Take $t_0\in \nn$ such that $2t_0+1$ is the largest odd number that's
smaller than $q/3$, then $t_0\geq (q-8)/6$. 
\[ \frac{q}{2(2t_0+1)} =\frac{3q}{2(q-5)}\leq \frac{48}{27}<2.\]
If $s\in [\frac{t_0q}{2t_0+1}, q/2]_\zz$, then $s=q/2-1$. 
The case $s\in [q/2, 2q/3]_\zz$ is treated similarly as the previous
cases. Except we will have $s$ lies in an interval of the form $[q/2,
c_tq]$, and the length of the interval is less than 1, so there are no
integers in it. 
\end{proof}

\begin{lem}\label{lem:not-isogenous-q-is-power-of-2}
  Suppose that $q=2^r\geq 16$. There does not exist a function
  $g\in \TZ_q$ such that $h= g+ g\circ\theta_s$ for any $s\in
  \umod{q}$. 
\end{lem}
\begin{proof}
  By Lemma~\ref{lem:q-is-a-power-of-2}, $s=1$ or $q/2-1$. We have
  observed in
  Subsection~\ref{subsec:uniqueness-of-g-in-case-q-is-power-of-2} that
  if $s=q/2-1$, then $g$ is given by (\ref{eq:30}).  But then it
  follows $g\circ \theta_s=g$, and $h=2g$, which is impossible, since
  there exists $a\in \umod{q}$ such that $h(a)=1$.
\end{proof}

\begin{lem}
  Let $q=2^r\geq 8$, $s=q/2-1\in \umod{q}$, and $\alpha=
  2i\sin(2\pi/q)$. Then $\qq(\alpha)$ is the subfield of $\qq(\zeta_q)$
  fixed by the subgroup $\dangle{s}$.
\end{lem}
\begin{proof}
  Indeed, we assume that $\zeta_q= \exp{2\pi i/ q}=\cos (2\pi/q)+i
  \sin(2\pi/q)$.  Then $\zeta_q^s= \exp{(q-2)\pi i/q} =-\cos (2\pi/q)+i
  \sin(2\pi/q)$. 

  Let $K$ be the subfield of $\qq(\zeta_q)$ fixed by $s$. It is clear
  that $2i\sin(2\pi/q)= \zeta_q+\zeta_q^s\in K$.
Now $\zeta_q$ satisfies the quadratic equation over
$\qq(\alpha)$: 
\[ x^2- 2i\sin(2\pi/q) x -1 =0. \]
We have $\qq(\alpha)\subseteq K \subset \qq(\zeta_q)$. We have just
shown $[\qq(\zeta_q): \qq(\alpha)]\leq 2$. However,
$[\qq(\zeta_q):K]\geq 2$. Therefore, $[\qq(\zeta_q): \qq(\alpha)]=2$,
and $\qq(\alpha)=K$. 
\end{proof}




\bibliographystyle{plain}
\bibliography{TeXBiB}
\end{document}